\documentclass[10pt]{amsart}
\usepackage{amsmath, amssymb}
\usepackage{dsfont}
\usepackage{amscd}
\usepackage[mathscr]{euscript} 
\usepackage[all]{xy}
\usepackage{url}
\usepackage{stmaryrd}
\usepackage{comment}
\usepackage[retainorgcmds]{IEEEtrantools}
\usepackage{hyperref}
\usepackage{orcidlink}
\usepackage{graphicx}
\usepackage{mathtools}
\usepackage[dvipsnames]{xcolor}
\usepackage{centernot}
\usepackage{marvosym}
\usepackage{tikz}
\usepackage{tikz-cd}
\usetikzlibrary{calc, arrows.meta, positioning}
\usepackage{xspace}
\usepackage{bbm}
\usepackage[title]{appendix}

\usepackage{tkz-fct}
\usepackage{soul}
\usepackage{lineno}
\usepackage[T1]{fontenc}
\usepackage[utf8]{inputenc}

\usepackage[english]{babel}
\usepackage[autostyle]{csquotes}
\usepackage[shortlabels]{enumitem}
\usepackage[backend=biber, style=alphabetic, sorting=nyt, backref=true]{biblatex}
\allowdisplaybreaks 

\title[Decompression and semigroup amenability]{Haar decompression and amenability \\ of Ellis flows}

\subjclass[2020]{Primary 37B02, Secondary 37A15, 54H15}
\keywords{Tame flows, Ellis semigroup, amenability, ergodic measures}

\author[D. M. HOFFMANN]{Daniel Max Hoffmann$^{\dagger}$\,
\orcidlink{0000-0002-4514-269X}}
\thanks{$^{\dagger}$SDG. The first author is supported
by the National Science Centre (Narodowe Centrum Nauki, Poland) 
grant no. 2021/43/B/ST1/00405.}
\address{$^{\dagger}$
Instytut Matematyki\\
Uniwersytet Warszawski\\
Warszawa\\
Poland}
\email{daniel.max.hoffmann@gmail.com}
\urladdr{{https://sites.google.com/site/danielmaxhoffmann/}}

\author[K. Krupi\'{n}ski]{Krzysztof Krupi\'{n}ski$^{\ddagger}$\,
\orcidlink{0000-0002-2243-4411}}
\address{$^{\ddagger}$Instytut Matematyczny Uniwersytetu Wroc{\l}awskiego, pl. Grunwaldzki 2, 50-384 Wroc{\l}aw, Poland}
\email{Krzysztof.Krupinski@math.uni.wroc.pl}

\date{\today}

 \DeclareMathOperator{\aut}{Aut} 
\DeclareMathOperator{\cl}{cl}

\DeclareMathOperator{\tp}{tp}

\DeclareMathOperator{\inv}{inv}
\DeclareMathOperator{\ev}{ev}

\DeclareMathOperator{\pr}{pr}

\DeclareMathOperator{\fGen}{fGen}

\DeclareMathOperator{\supp}{supp}
\DeclareMathOperator{\Prob}{Prob}
\DeclareMathOperator{\Erg}{Erg}

\newcommand{\unlhdm}{\mathrel{\unlhd_{\!m}}}

\newtheorem{theorem}{Theorem}[section]
\newtheorem{proposition}[theorem]{Proposition}
\newtheorem{lemma}[theorem]{Lemma}
\newtheorem{cor}[theorem]{Corollary}
\newtheorem{fact}[theorem]{Fact}
\theoremstyle{definition}
\newtheorem{definition}[theorem]{Definition}

\newtheorem{remark}[theorem]{Remark}
\newtheorem{question}[theorem]{Question}

\theoremstyle{remark}
\newtheorem*{theorem*}{Theorem}
\newtheorem*{cor*}{Corollary}

\theoremstyle{definition}
\theoremstyle{definition}

\theoremstyle{definition}

\theoremstyle{remark}

\newtheorem{clm}{Claim}
\newtheorem*{clm*}{Claim}

\AtEndEnvironment{proof}{\setcounter{clm}{0}}
\newenvironment{clmproof}[1][\proofname]{\proof[#1]}{\endproof}

\makeatletter
\providecommand*{\cupdot}{%
  \mathbin{%
    \mathpalette\@cupdot{}%
  }%
}
\newcommand*{\@cupdot}[2]{%
  \ooalign{%
    $\m@th#1\cup$\cr
    \sbox0{$#1\cup$}%
    \dimen@=\ht0 %
    \sbox0{$\m@th#1\cdot$}%
    \advance\dimen@ by -\ht0 %
    \dimen@=.5\dimen@
    \hidewidth\raise\dimen@\box0\hidewidth
  }%
}

\providecommand*{\bigcupdot}{%
  \mathop{%
    \vphantom{\bigcup}%
    \mathpalette\@bigcupdot{}%
  }%
}
\newcommand*{\@bigcupdot}[2]{%
  \ooalign{%
    $\m@th#1\bigcup$\cr
    \sbox0{$#1\bigcup$}%
    \dimen@=\ht0 %
    \advance\dimen@ by -\dp0 %
    \sbox0{\scalebox{2}{$\m@th#1\cdot$}}%
    \advance\dimen@ by -\ht0 %
    \dimen@=.5\dimen@
    \hidewidth\raise\dimen@\box0\hidewidth
  }%
}
\makeatother

\def\Ind#1#2{#1\setbox0=\hbox{$#1x$}\kern\wd0\hbox to 0pt{\hss$#1\mid$\hss}
\lower.9\ht0\hbox to 0pt{\hss$#1\smile$\hss}\kern\wd0}

\def\notind#1#2{#1\setbox0=\hbox{$#1x$}\kern\wd0
\hbox to 0pt{\mathchardef\nn=12854\hss$#1\nn$\kern1.4\wd0\hss}
\hbox to 0pt{\hss$#1\mid$\hss}\lower.9\ht0 \hbox to 0pt{\hss$#1\smile$\hss}\kern\wd0}

\newcommand{\CM}{{\mathcal M}}

\newcommand{\ext}{\mathrm{ext}}

\begin{document}

\begin{abstract}
Let $(X,G)$ be a tame flow and let $K$ be an Ellis group of its enveloping
semigroup $E(X,G)$. Although $K$ is a compact Hausdorff topological group in
its $\tau$-topology, the inclusion $K\hookrightarrow E(X,G)$ need not be
Borel. We show that normalized Haar measure on $K$ nevertheless determines,
via the Riesz--Markov theorem, a canonical regular Borel probability measure
$\mu_K$ on $E(X,G)$, called its Haar decompression.

Haar decompression gives an exact ergodicity description. For arbitrary
flow $(X,G)$, amenability of its Ellis flow is equivalent to hereditary
amenability of all finite powers $X^n$. 
In the tame setting 
$$\bigl(E(X,G),G\bigr)\text{ is amenable}
\quad\Longleftrightarrow\quad (X,G)\text{ is hereditarily amenable}.$$
Moreover, for the minimal left ideal $\CM$
in $E(X,G)$ containing $K$,
$\mu_K$ is $G$-invariant if and only if $(\CM,G)$ is amenable; when this holds, $\mu_K$ is the unique ergodic measure on $\CM$ and $\CM=\overline{K}$.
Then, we conclude that the ergodic measures of $E(X,G)$ are
precisely the Haar decompressions associated with amenable minimal left ideals.

We also prove that every minimal tame amenable flow is uniquely ergodic
and that every ergodic invariant measure on a tame flow has minimal
support. Consequently, every ergodic measure on a tame ambit
arises by evaluating a suitable Haar decompression.
\end{abstract}

\maketitle

\section{Introduction}
\subsection*{Haar decompression and its guiding principle}
Let $G$ be a discrete group. A $G$-flow is a compact Hausdorff space
equipped with an action of $G$ by homeomorphisms, and an ambit is a
pointed flow $(X,x_0,G)$ such that $\overline{Gx_0}=X$.

For a minimal equicontinuous flow, the unique invariant probability
measure is induced by the normalized Haar measure on the compact group
governing the action. In the paper, we ask how far this
compact-group origin of invariant measures survives beyond
equicontinuity. 
We show that it exists in tame dynamics, but in a less
direct form: the relevant compact groups are Ellis groups inside the
enveloping semigroup, and their Haar measures must first be regularized
before they can be transferred to the original flow. In this sense,
normalized Haar measures are the \emph{ergodic protoplasts} of
all ergodic, and in consequence of all invariant, measures. 

The enveloping semigroup $E(X,G)$ is the pointwise closure of $G$ in
$X^X$. Every minimal left ideal $\CM\unlhdm E(X,G)$ and every
idempotent $u^2=u\in\CM$ determine an Ellis group $K=u\CM$, equipped
with the $\tau$-topology \cite{Auslander,Glasner1976}. 
We work mainly with tame flows, whose
enveloping semigroups are characterized by the fragmentedness of their
elements; see
\cite{GlasnerMegrelishviliUspenskij,Glasner1,GM23,GM,KerrLi,Huang,CodHoff23}.
Tameness does not turn $E(X,G)$ into a topological semigroup, but it
provides enough measure-theoretic regularity to connect
its algebraic structure with invariant measures.

For a tame flow, every Ellis group $(K,\tau)$ is a compact Hausdorff
topological group
(Lemma~\ref{fact: RNC}, using \cite{Glasner2,CGK,BaZu}), and hence
carries normalized Haar measure $\mathfrak h_K$. The principal obstacle
is that the inclusion
$$j:(K,\tau)\longrightarrow E(X,G)$$
need not be Borel. Thus, the ordinary Borel pushforward
$j_*\mathfrak h_K$ might not work for our purposes. 
Appendix~\ref{sec:A1} gives a minimal metrizable tame flow for which $j$ is not Borel.

We prove nevertheless that, for every
$f\in C(E(X,G),\mathbb C)$, the restriction $f\circ j$ is measurable
for the completion of $\mathfrak h_K$. Consequently,
$$f\longmapsto\int_K f(j(k))\,d\mathfrak h_K(k)$$
defines a state on $C(E(X,G),\mathbb C)$. The Riesz--Markov theorem
therefore yields a unique regular Borel probability measure $\mu_K$ on
$E(X,G)$ representing this state. We call $\mu_K$ the
\emph{Haar decompression} associated with $K$
(Lemma~\ref{defprop:haar-decompression}). 
It follows from the definition that $\supp(\mu_K) \subseteq \overline{K} \subseteq \mathcal{M}$.

This construction leads to two complementary problems. The first is to
determine when $\mu_K$ is invariant under the left action of $G$ on
$E(X,G)$. The second is to determine which ergodic measures
on an ambit arise by evaluating Haar decompressions at its distinguished
point.

\subsection*{Main results}
Recall that a flow is \emph{amenable} if it carries an invariant
probability measure and is \emph{hereditarily amenable} if every
(nonempty) subflow is amenable. We say that $(X,G)$ is
\emph{finitely power-hereditarily amenable} if every subflow of every
finite power $X^n$ is amenable.

Our first result characterizes amenability of the Ellis flow:

\begin{theorem*}[Theorems~\ref{thm:ellis-flow-amenable} and
\ref{thm:hereditary-amenability-finite-powers}, and
Corollary~\ref{cor:ellis-amenability-characterization}]
For every flow $(X,G)$,
$$E(X,G)\text{ is amenable}
\quad\Longleftrightarrow\quad X\text{ is finitely power-hereditarily amenable}.$$
If $(X,G)$ is tame, these conditions are further equivalent to
hereditary amenability of $(X,G)$.
\end{theorem*}

The first equivalence above reflects the finite-coordinate nature of
continuous functions on $E(X,G)$: functions depending on finitely many
coordinates form a uniformly dense subalgebra of
$C(E(X,G),\mathbb C)$. Then, the tame part of the theorem above
shows that ordinary
hereditary amenability already passes to all finite powers. Its
proof goes first through metrizable tame factors and then removes
separability. Moreover, amenability of the Ellis flow is automatically hereditary
(Remark~\ref{rem:ellis-amenability-is-hereditary}).

The above theorem decides when the Ellis flow admits some invariant
probability measure. On the other hand, Haar decompression gives a corresponding
criterion at the level of an individual minimal left ideal:

\begin{theorem*}[Theorem~\ref{thm:haar-decompression-amenability} and
Remark~\ref{rem:invariant-decompression-ergodic}]
Let $(X,G)$ be tame, let $\CM\unlhdm E(X,G)$, let
$u^2=u\in\CM$, and put $K=u\CM$. Then
$$
\mu_K\text{ is $G$-invariant}
\quad\Longleftrightarrow\quad
(\CM,G)\text{ is amenable}.
$$
When these equivalent conditions hold, $\mu_K$, regarded as a measure
on $\CM$, is its unique invariant probability measure. Moreover,
$$
\mu_K\in\Erg_G\bigl(E(X,G)\bigr),
\qquad
\supp(\mu_K)=\CM=\overline K.
$$
\end{theorem*}
Here, for a flow $(Y,G)$, $\Erg_G(Y)$ denotes the collection of all ergodic $G$-invariant regular Borel probability measures on $Y$.
Together with the minimal-support result below, the preceding theorem
gives an exact description of the ergodic measures on the Ellis flow:
$$\Erg_G\big(E(X,G)\big)=\left\{\mu_{u\CM}:\CM\unlhd_m E,\quad u^2=u\in\CM,\quad
(\CM,G)\text{ is amenable} \right\}.$$
If $(X,G)$ is hereditarily amenable, then every minimal left ideal is
amenable, so all Haar decompressions occur in the above description.
The tameness assumption is essential:
Appendix~\ref{sec:A3} constructs a metrizable hereditarily amenable flow
whose Ellis flow is not amenable.

A second group of results establishes rigidity of invariant measures on
tame flows:

\begin{theorem*}[Theorem~\ref{thm:nonmetric-unique-ergodicity},
Proposition~\ref{prop:tame-ergodic-support-minimal}, and
Corollary~\ref{cor:ergodic-measures-minimal-subflows}]
Every minimal tame amenable flow is uniquely ergodic. Moreover, if
$(X,G)$ is tame, then the support of every
$\mu\in\Erg_G(X)$ is a minimal subflow, and the support map
$$
\Erg_G(X)\longrightarrow
\operatorname{Min}^{\mathrm{am}}_G(X),
\qquad
\mu\longmapsto\supp(\mu),
$$
is a bijection onto the collection of amenable minimal subflows of
$X$. In particular, distinct ergodic invariant measures have disjoint
supports.
\end{theorem*}

In the metrizable minimal case, unique ergodicity follows from the
almost one-to-one structure theorem for minimal tame flows
\cite{GLASNER_2007,Glasner2} and the regularity theorem of
\cite{FGJO21}. We also identify the unique invariant measure directly
as the evaluation of every Haar decompression. A separable
$C^*$-algebra reduction then yields unique ergodicity without
metrizability. The minimal-support theorem uses a related reduction
together with an integration identity supplied by tameness, which
plays the role of dominated convergence for limits indexed by general
nets. A special case of the minimal-support phenomenon for tame
semicascades, under additional assumptions, was obtained in
\cite[Theorem~3.6]{Romanov16}.

We now make the ergodic protoplast principle precise. For an ambit
$(X,x_0,G)$, evaluation at the distinguished point is the continuous
$G$-map
$$\ev_{x_0}:E(X,G)\longrightarrow X,\qquad\eta\longmapsto\eta(x_0).$$
There are two notions of ergodic transfer. We say
that an ambit $(X,x_0,G)$ has the \emph{ergodic Ellis-transfer
property} if
$$(\ev_{x_0})_*\mu_K\in\Erg_G(X)$$
for every Ellis group $K$ of $E(X,G)$. We say that it has the
\emph{ergodic Ellis-realization property}
if every
$\mu\in\Erg_G(X)$ is equal to $(\ev_{x_0})_*\mu_K$ for at least one
Ellis group $K$.

\begin{theorem*}[Corollaries~\ref{cor:ergodic-ellis-realization} and \ref{cor:ergodic-haar-transfer}]
Every tame ambit has the ergodic Ellis-realization property, and every tame hereditarily amenable ambit has the ergodic Ellis-transfer property.
\end{theorem*}

\subsection*{Examples}
The Appendix shows that our results are sharp. 
Appendix~\ref{sec:A1} gives a
minimal metrizable tame flow for which
$j:(K,\tau)\to E(X,G)$ is not Borel. 
Appendix~\ref{sec:A3}
gives a counterexample to the hereditary-amenability
characterization of the Ellis flow, i.e. 
a hereditarily amenable non-tame flow with non-amenable Ellis flow.

In the split-circle dihedral flow
of Appendix~\ref{subsec:split-circle-ordinary}, two Ellis groups are
interchanged by the action, whereas their Haar decompressions coincide
and are invariant; decompression therefore restores invariance which is
not visible on either group separately.

\subsection*{Model-theoretic motivation}
The construction is motivated by model-theoretic dynamics. Enveloping
semigroups of type-space flows have played a central role since the work of
Newelski \cite{newelski09,newelski12}. Model-theoretic
tameness (NIP) corresponds to dynamical tameness
\cite{Simon_2015,Iba16,KrRz,CodHoff23}, and Ellis groups are closely related to canonical
compact quotients such as $G/G^{00}$ and the Kim--Pillay Galois group; see
\cite{KruPil17,ArtemPierre,Krupinski2019,Krupinski2019a,KruPiRze,HruKruPi}. 

In
\cite{HRz}, jointly with Rzepecki, following ideas of \cite{ArtemPierre}, the first author described ergodic measures of
$(S_{\bar m}(M),\aut(M))$ for countable NIP theories in terms of the Haar measure on the Kim--Pillay Galois group. The present paper isolates 
a similar, purely dynamical mechanism: the Haar measure on an Ellis group is first regularized on the enveloping semigroup and then transferred by evaluation 
(cf. Subsection \ref{subsec:model.theory}).

The idea of inducing a measure on the enveloping semigroup from the Haar measure on an Ellis group appeared already in \cite[Section 4]{CGK} in the context of definable groups, but the situation there did not require regularization which we need in the abstract context (cf. Subsection \ref{subsec:model.theory}).

\subsection*{Organization of the paper}
Section~\ref{sec:preliminaries} fixes the conventions and recalls the
required facts about Ellis semigroups, the
$\tau$-topology, fragmented maps, and tame flows.
Section~\ref{sec:haar-decompression} establishes the required
measurability and constructs Haar decompressions.
Section~\ref{sec:unique.minimal} proves unique ergodicity of minimal
tame amenable flows and minimality of supports of ergodic measures.
Section~\ref{sec:invariance} proves the general finite-power
amenability criterion, its tame hereditary-amenability consequence,
and the invariance criterion for individual Haar decompressions.
The final section develops ergodic Ellis transfer and realization, proves the ergodic Ellis-realization property for every tame ambit, obtains the Choquet representation in the metrizable hereditarily amenable case and links our general results with the model-theoretic dynamics.
Appendix~\ref{sec:A} contains the non-Borel inclusion example, the
split-circle computation, and the counterexample showing that tameness
is essential.

\section{Preliminaries}
\label{sec:preliminaries}
Throughout, $G$ denotes a discrete group. All compact spaces are assumed to be Hausdorff, all probability measures are regular Borel measures.

\subsection{Flows, invariant measures and Gelfand duality}
\label{subsec:flows-and-invariant-measures}
A $G$-\emph{flow} is a compact space $X$ equipped with an action of $G$ by homeomorphisms. A \emph{subflow} is a nonempty closed $G$-invariant subset, and a \emph{factor map} is a continuous surjective $G$-map. A pointed flow $(X,x_0,G)$ is an \emph{ambit} if
$\overline{Gx_0}=X$, and a flow is \emph{minimal} if every orbit is dense.

For a compact space $X$, let $C(X):=C(X,\mathbb C)$, endowed with the
uniform norm and the involution $f^*(x):=\overline{f(x)}$.
Its positive cone is
\begin{align*}
C(X)_+
&:=
\{h^*h:h\in C(X)\}\\
&=
\{f\in C(X):f=f^*\text{ and }f(x)\geqslant0
\text{ for every }x\in X\}.
\end{align*}
A \emph{state} on a unital $C^*$-algebra $A$ is a complex-linear
functional $\Lambda:A\to\mathbb C$ such that
$\Lambda(a^*a)\geqslant 0$ for every $a\in A$ and
$\Lambda(1_A)=1$.

We write $\Prob(X)$ for the compact convex space of regular Borel
probability measures on $X$, equipped with the weak-$*$ topology inherited
from $C(X)^*$. If $\pi:X\to Y$ is Borel and $\mu\in\Prob(X)$, then
$\pi_*\mu$ denotes the pushforward of $\mu$. For $g\in G$, we write
$g_*\mu$ for the pushforward under $x\mapsto gx$, and put
$$\Prob_G(X):=\{\mu\in\Prob(X):g_*\mu=\mu\text{ for every }g\in G\}.$$
An invariant measure is \emph{ergodic} if it is an extreme point of
$\Prob_G(X)$, and we write
$$\Erg_G(X):=\operatorname{ext}\bigl(\Prob_G(X)\bigr).$$
The flow $(X,G)$ is \emph{amenable} if $\Prob_G(X)\neq\varnothing$, and it
is \emph{uniquely ergodic} if $\Prob_G(X)$ is a singleton. It is
\emph{hereditarily amenable} if every subflow of $X$ is amenable. The action
of $G$ on $C(X)$ is given by $(g\cdot f)(x):=f(g^{-1}x)$.

We briefly recall the form of Gelfand duality used below 
(see, for example, \cite{MurphyCstar}). 
For a unital commutative
$C^*$-algebra $A$, its spectrum $\operatorname{Spec}(A)$ is the set of
characters, i.e. unital $*$-homomorphisms $\chi:A\to\mathbb C$, equipped with the
weak-$*$ topology. It is compact Hausdorff, and the Gelfand transform
identifies $A$ isometrically with
$C(\operatorname{Spec}(A),\mathbb C)$. If
$A\subseteq C(Y,\mathbb C)$ is a unital $C^*$-subalgebra, then
$$
q_A:Y\longrightarrow\operatorname{Spec}(A),
\qquad
q_A(y)(a):=a(y),
$$
is a continuous surjection; it is equivariant whenever $A$ is
invariant under the acting group. An inclusion $A\subseteq B$ induces
the continuous surjection
$\operatorname{Spec}(B)\to\operatorname{Spec}(A)$ given by restriction
of characters. Moreover, $\operatorname{Spec}(A)$ is metrizable when
$A$ is separable.
Moreover, if $A$ is a subset of a $C^*$-algebra $\mathcal{C}$,
then $C^*(A)$ stands for the $C^*$-subalgebra of $\mathcal{C}$ generated by a $A$.

We also recall the barycentric terminology used in the proof of Lemma \ref{lemma:tameness.dominated.convergence}. 
Let $Q$ be a compact convex subset of a locally convex real topological vector
space. Every $\Lambda\in\Prob(Q)$ has a unique \emph{barycenter}
$b(\Lambda)\in Q$, characterized by
$$f\bigl(b(\Lambda)\bigr)=\int_Q f(q)\,d\Lambda(q)$$
for every continuous affine function $f:Q\to\mathbb R$. 
A universally measurable function $F:Q\to\mathbb R$ is said to satisfy the
\emph{barycentric formula} if
\[
F\bigl(b(\Lambda)\bigr)=\int_Q F(q)\,d\Lambda(q)
\qquad\text{for every }\Lambda\in\Prob(Q).
\]
By \cite[Theorem~4.21]{LMNS10}, every real-valued fragmented affine
function on a compact convex set satisfies the barycentric formula.

\subsection{Ellis semigroups, fragmented maps and tameness}
The \emph{enveloping semigroup} of a flow $(X,G)$ is
$$E(X,G):=\overline{\{g_X:g\in G\}}^{\,X^X},\qquad g_X(x):=gx,$$
where $X^X$ carries the product topology. Multiplication is composition: $(pq)(x):=p(q(x))$.
For every fixed $q\in E(X,G)$, the map $p\mapsto pq$ is continuous. For
$x\in X$, evaluation is the continuous map
$$\ev_x:E(X,G)\longrightarrow X,
\qquad
\ev_x(p):=p(x),$$
and $\ev_x[E(X,G)]=\overline{Gx}$.
A factor map $\pi:X\to Y$ induces a continuous surjective
$G$-equivariant semigroup homomorphism
$$\pi_E:E(X,G)\longrightarrow E(Y,G)$$
satisfying $\pi\circ p=\pi_E(p)\circ\pi$ for every $p\in E(X,G)$.

Put $E:=E(X,G)$. We write $\CM\unlhdm E$ for a minimal left ideal.
If $u^2=u\in\CM$, then
$K:=u\CM$ is the corresponding \emph{Ellis group}.
For $p\in E$ and $A\subseteq E$, define the Ellis circle operation by
\begin{IEEEeqnarray*}{rCl}
    p\circ A &:=& \bigl\{q\in E:\text{there are nets }(g_i)_i\subseteq G\text{ and }(a_i)_i\subseteq A \\
 & & \text{ such that }g_i\to p\text{ and }g_i a_i\to q\bigr\}.
\end{IEEEeqnarray*}
For $A\subseteq K$, put
$$\cl_\tau(A):=K\cap(u\circ A)=u(u\circ A).$$
The operator $\cl_\tau$ is a closure operator on $K$. The topology
determined by this closure operator is called the
\emph{$\tau$-topology}.
The $\tau$-topology is coarser than the topology on $K$ inherited
from $E$. Moreover, $(K,\tau)$ is always a quasi-compact
$T_1$ semitopological group, that is, multiplication is separately
continuous. No Hausdorffness is asserted here. See
\cite{Glasner1976} for the classical construction in the
$\beta G$ setting, and
\cite[Appendix A]{rzepecki2018}
for the general formulation used here.
Whenever $(K,\tau)$ is a compact Hausdorff topological group,
$\mathfrak h_K$ denotes its normalized Haar measure.

\begin{fact}
\label{fact:ellis-group-functoriality}
Let $(X,G)$ and $(Y,G)$ be flows, and let $\Phi:E(X,G)\longrightarrow E(Y,G)$ be a continuous surjective $G$-equivariant semigroup homomorphism.
Let $\CM\unlhdm E(X,G)$, $u^2=u\in\CM$,
and put $\mathcal N:=\Phi[\CM]$, $v:=\Phi(u)$, $K:=u\CM$, $L:=v\mathcal N=\Phi[K]$.
Then $\mathcal N\unlhdm E(Y,G)$, $v^2=v\in\mathcal N$,
and the restriction
$$
\Phi|_K:(K,\tau)\longrightarrow(L,\tau)
$$
is a continuous group epimorphism and a quotient map for the
$\tau$-topologies.
(In particular, the conclusion applies when $\Phi=\pi_E$ is induced
by a factor map $\pi:X\to Y$, and when $Y\subseteq X$ is a subflow
and $\Phi:E(X,G)\to E(Y,G)$ is the restriction map.)
If $(K,\tau)$ and $(L,\tau)$ are compact Hausdorff topological groups,
then $(\Phi|_K)_*\mathfrak h_K=\mathfrak h_L$.
\end{fact}

\begin{proof}
The assertions concerning minimal left ideals, idempotents, Ellis
groups, and the quotient property are
\cite[Facts~2.3 and~2.5]{KLM21}. The last assertion follows from the
uniqueness of normalized Haar measure on a compact Hausdorff
topological group.
\end{proof}

For \(\eta\in E(X,G)\) and
\(\bar x=(x_1,\ldots,x_n)\in X^n\), we write
$$\eta(\bar x):=\bigl(\eta(x_1),\ldots,\eta(x_n)\bigr).$$

Let $X$ be a compact space and let $Y$ be a uniform space. A map
$f:X\to Y$ is \emph{fragmented} if, for every nonempty set
$A\subseteq X$ and every entourage $V$ of $Y$, there is a nonempty
relatively open set $U\subseteq A$ such that
$$(f(x),f(y))\in V$$
for all $x,y\in U$.
When $Y$ is metrizable, this is equivalent to requiring that, for every
$\varepsilon>0$ and every nonempty closed $A\subseteq X$, there is a
nonempty relatively open $U\subseteq A$ with
$$\operatorname{diam}f[U]<\varepsilon.$$
In particular, a bounded function $f:X\to\mathbb C$ is fragmented if
and only if both $\operatorname{Re}f$ and $\operatorname{Im}f$ are
fragmented.
A flow $(X,G)$ is \emph{tame} if every element of $E(X,G)$ is a
fragmented self-map of $X$. This is one of the standard equivalent
definitions of tameness; see \cite[Definition~3.1]{GM23}. We use this
characterization throughout.

We equip every compact Hausdorff space with its unique compatible
uniformity. A flow $(X,G)$ is \emph{equicontinuous} if the family
$$
\{g_X:g\in G\}\subseteq X^X
$$
is equicontinuous with respect to this uniformity, where $g_X(x):=gx$.

\begin{lemma}
\label{lem:finite-powers-of-tame-flows}
Let $(X,G)$ be a flow.
\begin{enumerate}[(1)]
\item If $(X,G)$ is equicontinuous, then, for every cardinal $\kappa$,
the flow $X^\kappa$, equipped with the diagonal $G$-action, is
equicontinuous.

\item If $(X,G)$ is tame, then, for every cardinal $\kappa$, the flow
$X^\kappa$, equipped with the diagonal $G$-action, is tame.

\item If $(X,G)$ is tame, then its enveloping semigroup $E(X,G)$,
equipped with the natural left $G$-action, is tame.
\end{enumerate}

Consequently, if $X$ is metrizable and $(X,G)$ is tame, then every
subflow of every finite power $X^n$ is metrizable and tame.
\end{lemma}

\begin{proof}
(1) Let $V$ be an entourage of $X^\kappa$. By the definition of the
product uniformity, there are a finite set $F\subseteq\kappa$ and
entourages $V_\alpha$ of $X$, for $\alpha\in F$, such that
$$
(x_\alpha,y_\alpha)\in V_\alpha
\text{ for every }\alpha\in F
\quad\Longrightarrow\quad
(x,y)\in V.
$$
For every $\alpha\in F$, equicontinuity of $(X,G)$ gives an entourage
$U_\alpha$ of $X$ such that
$$
(s,t)\in U_\alpha
\quad\Longrightarrow\quad
(gs,gt)\in V_\alpha
\text{ for every }g\in G.
$$
The basic entourage of $X^\kappa$ determined by the $U_\alpha$,
$\alpha\in F$, witnesses equicontinuity of the diagonal action on
$X^\kappa$.

For (2) and (3), see
\cite[Lemma~5.4 and the sentence immediately following its proof]
{GlasnerMegrelishvili2012}.
\end{proof}

\begin{fact}[Theorem~8.2 in \cite{GlasnerMegrelishvili2013}]\label{fact:tame-injective}
If $(Y,G)$ is tame, then the canonical restriction homomorphism
$$\operatorname{res}:E(\operatorname{Prob}(Y),G)\longrightarrow E(Y,G)$$
is a topological semigroup (and $G$-flow) isomorphism. 
We set $\eta^\#:=\operatorname{res}^{-1}(\eta)$ for every $\eta\in E(Y,G)$.
\end{fact}

For further background on tame flows and Ellis semigroups, see
\cite{Auslander,Glasner1976,Glasner1,GM23,CodHoff23}.

\section{Haar decompression on the Ellis semigroup}
\label{sec:haar-decompression}
We first establish the measurability input needed to define Haar
decompression. The argument combines a descent property for fragmented
functions with the measurability of bounded fragmented functions for the completion of every regular Borel probability measure.

\begin{lemma}\label{lem:descent-complex}
Let $\pi:L\to K$ be a continuous surjection of compact Hausdorff spaces.
Let $f:K\to\mathbb C$ be bounded. If $f\circ\pi:L\to\mathbb C$ is
fragmented, then $f$ is fragmented.
\end{lemma}

\begin{proof}
This is a standard fact (see, for example, \cite[Lemma 2.4 ]{Simon_2015}).
Suppose that $f$ is not fragmented. Then there are $\varepsilon>0$ and a
nonempty closed set $F\subseteq K$ such that
$\operatorname{diam} f[U]\geqslant \varepsilon$
for every nonempty relatively open subset $U\subseteq F$.

By Zorn's lemma choose a minimal closed set
$L_0\subseteq \pi^{-1}(F)$
such that $\pi[L_0]=F$: the existence follows because if
$(L_i)_{i\in I}$ is a decreasing chain of closed subsets of $\pi^{-1}(F)$
mapping onto $F$, then for each $x\in F$ the compact sets
$L_i\cap \pi^{-1}(\{x\})$ have the finite intersection property, so
$\bigcap_{i\in I} L_i$
still maps onto $F$.
Since $f\circ\pi$ is fragmented on $L$, and hence also on the closed subspace
$L_0$, there is a nonempty relatively open set $V\subseteq L_0$ such that $\operatorname{diam} (f\circ\pi)[V]<\varepsilon$.
If $L_0\setminus V$ still mapped onto $F$, this would contradict the minimality
of $L_0$. Hence,
$$U:=F\setminus \pi[L_0\setminus V]$$
is nonempty. Moreover, since $L_0\setminus V$ is compact, $\pi[L_0\setminus V]$
is closed in $F$, so $U$ is relatively open in $F$.

For every $x\in U$, $L_0\cap\pi^{-1}(\{x\})\subseteq V$.
Therefore, $f[U]\subseteq (f\circ\pi)[V].$
Consequently,
$$\operatorname{diam} f[U]
        \leqslant
        \operatorname{diam} (f\circ\pi)[V]
        <\varepsilon,$$
contradicting the choice of $F$. Thus, $f$ is fragmented.
\end{proof}

By \cite[Theorem~A.121]{LMNS10}, bounded fragmented functions on compact
spaces are \(\Sigma_1(H_s)\)-measurable, where \(H_s\) denotes the family of resolvable sets. By \cite[Proposition~A.118]{LMNS10}, resolvable sets are universally measurable. Thus, the following fact is standard. For the convenience of the reader, we include a self-contained proof.

\begin{fact}\label{fact:univ}
Let \(K\) be a compact Hausdorff space and let \(f:K\to\mathbb C\) be bounded
and fragmented. Then, for every regular Borel probability measure \(\lambda\) on \(K\),
there are a Borel function \(f_\lambda:K\to\mathbb C\) and a Borel set
\(N_\lambda\subseteq K\) with \(\lambda(N_\lambda)=0\) such that
    $$f=f_\lambda\quad\text{on }K\setminus N_\lambda .$$
In particular, $f$ is $\lambda$-measurable (i.e., measurable for the $\lambda$-completion of the Borel $\sigma$-algebra on $K$). Consequently, if $K$ is a compact Hausdorff group, then $f$ is measurable for the completion
of the Borel $\sigma$-algebra on $K$ with respect to the normalized Haar measure on $K$.
\end{fact}

\begin{proof}
For every
\(\varepsilon>0\) and every nonempty closed subset \(F\subseteq K\), there is
a nonempty relatively open subset \(V\subseteq F\) such that
$\operatorname{diam} f[V]<\varepsilon$.
Fix \(\varepsilon>0\). We construct, by transfinite recursion, a decreasing family of closed sets \((F_\gamma)_\gamma\) and disjoint Borel pieces \((P_\gamma)_\gamma\). 
Put \(F_0:=K\). If \(F_\gamma\neq\varnothing\), choose an
open set \(O_\gamma\subseteq K\) such that
$$
        P_\gamma:=F_\gamma\cap O_\gamma\neq\varnothing
        \quad\text{and}\quad
        \operatorname{diam} f[P_\gamma]<\varepsilon ,
$$
and set
\[
        F_{\gamma+1}:=F_\gamma\setminus P_\gamma
        =F_\gamma\cap (K\setminus O_\gamma).
\]
At limit ordinals \(\delta\), put $F_\delta:=\bigcap_{\gamma<\delta}F_\gamma$.
Since the closed subsets of \(K\) form a set, this strictly decreasing process
stops: there is an ordinal \(\theta\) such that \(F_\theta=\varnothing\). Hence,
\[
        K=\bigsqcup_{\gamma<\theta}P_\gamma .
\]
Indeed, if \(x\in K\), then the least ordinal \(\gamma\) such that
\(x\notin F_\gamma\) cannot be a limit ordinal, and therefore
\(\gamma=\delta+1\) with \(x\in P_\delta\). Each \(P_\gamma\) is Borel, being
the intersection of a closed set and an open set.

Now, fix a regular Borel probability measure \(\lambda\) on \(K\). We first recall
the following continuity property for regular Borel probability measures on compact
spaces: if \((C_i)_{i\in I}\) is a downward directed family of closed subsets of
\(K\), then
$$\lambda\Big(\bigcap_{i\in I} C_i\Big)=\inf_{i\in I}\lambda(C_i).$$
For \(\gamma\leqslant\theta\), put
$$s_\gamma:=\sum_{\delta<\gamma}\lambda(P_\delta),$$
where the sum means the supremum of all finite subsums. Since the sets
\(P_\delta\) are pairwise disjoint and \(\lambda(K)=1\), only countably many
summands are positive, so this agrees with the usual countable sum over those
positive summands. We claim that
$\lambda(F_\gamma)=1-s_\gamma$, $\text{for all }\gamma\leqslant\theta$.
This is proved by transfinite induction. The case \(\gamma=0\) is immediate.
For the successor step, use
$F_\gamma=F_{\gamma+1}\sqcup P_\gamma$,
thus
$$\lambda(F_{\gamma+1})=\lambda(F_\gamma)-\lambda(P_\gamma)=1-s_\gamma-\lambda(P_\gamma)=1-s_{\gamma+1}.$$
At a limit ordinal \(\delta\), the closed sets \((F_\gamma)_{\gamma<\delta}\) are
downward directed and
$F_\delta=\bigcap_{\gamma<\delta}F_\gamma$.
By the continuity property given above,
$$\lambda(F_\delta)=\lambda\Big(\bigcap_{\gamma<\delta}F_\gamma\Big)
=\inf_{\gamma<\delta}\lambda(F_\gamma)=\inf_{\gamma<\delta}(1-s_\gamma)
=1-\sup_{\gamma<\delta}s_\gamma=1-s_\delta.$$
Since \(F_\theta=\varnothing\), we obtain
$\sum_{\gamma<\theta}\lambda(P_\gamma)=1$.

Let
$$\Gamma_\varepsilon=\{\gamma<\theta:\lambda(P_\gamma)>0\}.$$
This set is countable: for each \(n\geqslant 1\), there are only finitely many
\(\gamma\) with \(\lambda(P_\gamma)\geqslant 1/n\). Put
$$B_\varepsilon:=\bigcup_{\gamma\in\Gamma_\varepsilon}P_\gamma.$$
Then \(B_\varepsilon\) is Borel and \(\lambda(B_\varepsilon)=1\).
For each \(\gamma\in\Gamma_\varepsilon\), choose a point
\(x_\gamma\in P_\gamma\), and set
$$c_\gamma:=f(x_\gamma).$$
Define a Borel step function \(g_\varepsilon:K\to\mathbb C\) by
$$g_\varepsilon=\sum_{\gamma\in\Gamma_\varepsilon}
c_\gamma\,\mathbf 1_{P_\gamma},$$
with value \(0\) off \(B_\varepsilon\). Since the sets \(P_\gamma\) are pairwise
disjoint and \(\Gamma_\varepsilon\) is countable, \(g_\varepsilon\) is Borel.
For \(x\in B_\varepsilon\), say \(x\in P_\gamma\), we have
$$|f(x)-g_\varepsilon(x)|=|f(x)-c_\gamma|<\varepsilon,$$
because \(\operatorname{diam} f[P_\gamma]<\varepsilon\).

Apply the construction with \(\varepsilon=1/n\). Put $N:=\bigcup_{n=1}^{\infty}(K\setminus B_{1/n})$.
Then \(N\) is Borel and \(\lambda(N)=0\). Let
$$h_n:=g_{1/n}\,\cdot\,\mathbf 1_{K\setminus N}.$$
The functions \(h_n\) are Borel, and for every \(x\in K\setminus N\),
$$|h_n(x)-f(x)|<\frac1n .$$
Thus, \(h_n\to f\) pointwise on \(K\setminus N\), while \(h_n=0\) on \(N\) for
all \(n\). Consequently,
$$f_\lambda:=\lim_{n\to\infty}h_n$$
is Borel and agrees with \(f\) on \(K\setminus N\). Therefore, \(f\) is measurable
for the \(\lambda\)-completion of the Borel \(\sigma\)-algebra.
\end{proof}

Let $(X,G)$ be a tame flow, $u^2=u\in\CM\unlhd_m E:=E(X,G)$, $K:=u\CM$,
let $j:K\to E$ be the inclusion map.
The next lemma is essentially a general abstract form of the so-called {\em Revised Newelski's Conjecture} stated in \cite{KruPi23} for definable groups, proved in \cite{CGK} for countable definable groups, and then in \cite{BaZu} in a general abstract context (and without the countability assumption) for minimal flows. Here, we remove the assumption on minimality of the flow.

\begin{lemma}\label{fact: RNC}
    $K$ is a compact Hausdorff group in the $\tau$-topology.
\end{lemma}

\begin{proof}
    By Lemma~\ref{lem:finite-powers-of-tame-flows}(3), the flow
$(E(X,G),G)$ is tame, and hence its minimal subflow $(\CM,G)$ is tame.
By \cite[Lemma~5.16]{CGK}, $(K,\tau)$ is isomorphic as a
semitopological group to an Ellis group $K'$ of $E(\CM,G)$. Since
$(\CM,G)$ is minimal and tame, \cite[Theorem~11.7]{BaZu} implies that
$K'$ is Hausdorff, and therefore so is $K$. 

By the basic properties of the $\tau$-topology recalled in
Section~\ref{sec:preliminaries}, $(K,\tau)$ is a quasi-compact
$T_1$ semitopological group.
Since by the first paragraph $K$ is Hausdorff, the Ellis joint continuity theorem implies that it is a compact Hausdorff topological group.
\end{proof}

\begin{lemma}\label{lem:Fj.fragmented}
If $F\in C(E,\mathbb{C})$,
then $F\circ j:K\to \mathbb{C}$ is a fragmented map with respect to the $\tau$-topology on $K$.
\end{lemma}

\begin{proof}
    Consider the map $\lambda_u:E\to E$ given by $\lambda_u(\eta):=u\eta$. By Lemma~\ref{lem:finite-powers-of-tame-flows}(3), the flow $(E(X,G),G)$ is tame. Moreover, $\lambda_u\in E(E(X,G),G)$. Therefore, $\lambda_u:E\to E$ is fragmented.
    In particular, its restriction $\lambda_u|_{\overline{K}}:\overline{K}\to\overline{K}$ is a fragmented map, and thus $F\circ\lambda_u|_{\overline{K}}$ is fragmented as well.
    
    Because $(X,G)$ is tame, we have that $(K,\tau)$ is a compact Hausdorff topological group by Lemma~\ref{fact: RNC}.
    By Lemma 3.1 from \cite{KruPiRze}, the map $\lambda_u |_{\bar{K}}:\overline{K}\to K$ is a continuous surjection (from the induced topology to the $\tau$-topology).

    Now, the composition $(F\circ j)\circ\lambda_u|_{\overline{K}}$
    is equal to $F\circ\lambda_u|_{\overline{K}}$, hence it is fragmented. We apply Lemma~\ref{lem:descent-complex}, to get the conclusion. 
\end{proof}

Let $\mathfrak{h}_K$ be the normalized Haar measure on the compact Hausdorff group $K$.

\begin{cor}\label{cor:Fj.measurable}
If $F\in C(E,\mathbb{C})$, then $F\circ j:K\to \mathbb{C}$ is $\mathfrak{h}_K$-measurable.
\end{cor}

\begin{proof}
    By Lemma~\ref{lem:Fj.fragmented}, we have that $F\circ j:K\to\mathbb{C}$ is fragmented.
    The conclusion follows by Fact~\ref{fact:univ}.
\end{proof}

\begin{fact}[Riesz--Markov representation]\label{fact:riesz-complex}
Let \(Y\) be a compact Hausdorff space. If
$$\Lambda:C(Y,\mathbb C)\to\mathbb C$$
is a positive unital complex-linear functional, then there is a unique regular
Borel probability measure \(\mu_\Lambda\) on \(Y\) such that
$$\Lambda(F)=\int_Y F\,d\mu_\Lambda\quad\text{for every }F\in C(Y,\mathbb C).$$
\end{fact}

\begin{lemma}\label{defprop:haar-decompression}
The formula
$$
        \Lambda_K(F)
        :=
        \int_K F(j(k))\,d\mathfrak h_K(k)
        \qquad(F\in C(E,\mathbb C))
$$
defines a positive unital complex-linear functional on \(C(E,\mathbb C)\).
Consequently, by Fact \ref{fact:riesz-complex}, there is a unique regular Borel
probability measure \(\mu_K\) on \(E\) such that
$$
        \int_E F\,d\mu_K
        =
        \int_K F(j(k))\,d\mathfrak h_K(k)
        \quad\text{for every }F\in C(E,\mathbb C).
$$
We call \(\mu_K\) the decompression of \(\mathfrak h_K\).
\end{lemma}

\begin{proof}
The integral is well defined by Corollary \ref{cor:Fj.measurable}. Complex linearity follows immediately
from linearity of the integral. Moreover,
$$\Lambda_K(1_E)=\int_K 1\,d\mathfrak h_K=1.$$
If \(F\in C(E)_+\), 
then $F(j(k)) \geqslant 0$ for every $k \in K$.
Therefore, \(F\circ j\) is real-valued and non-negative, and
$$\Lambda_K(F)=\int_K F(j(k)) \,d\mathfrak h_K(k)\geqslant 0 .$$
Thus, \(\Lambda_K\) is positive. 
The Riesz--Markov representation theorem now yields the unique regular Borel probability measure \(\mu_K\) with the displayed property.
\end{proof}

\begin{remark}\label{rem:no-raw-pushforward}
The measure $\mu_K$ is not defined as a raw pushforward. Although
$F\circ j$ is $\mathfrak h_K$-measurable for every $F\in C(E)$, the
inclusion $j:(K,\tau)\to E$ need not be Borel; see
Appendix~\ref{sec:A1}. It is not known whether $j$ is always measurable
from the $\mathfrak h_K$-completion of the Borel $\sigma$-algebra on
$K$ to the Borel $\sigma$-algebra on $E$. Even when this measurability
holds, it is not known whether the Borel pushforward
$j_*\mathfrak h_K$ must be regular. The Riesz--Markov construction
instead produces a canonical regular Borel probability measure on $E$.
If $j$ is measurable in the above sense and $j_*\mathfrak{h}_K$ is
regular, then
$$
\mu_K=j_*\mathfrak{h}_K,
$$
since the two measures have the same integrals against every function
in $C(E)$.
\end{remark}

\begin{question}\label{question:39}
\begin{enumerate}
    \item Does there exist a tame flow $(X,G)$ for which the inclusion
$j:(K,\tau)\longrightarrow E(X,G)$
is not measurable with respect to the $\mathfrak h_K$-completion?

\item Suppose that $j$ is measurable for the $\mathfrak h_K$-completion.
Must the Borel pushforward $j_*\mathfrak h_K$ be regular?
\end{enumerate}
\end{question}

\begin{remark}\label{remark: support of the compression}
\begin{enumerate}
\item $\supp(\mu_K)\subseteq\overline{u\CM}\subseteq\CM$.
\item If $\mu_K$ is $G$-invariant, then $\overline{u\CM}=\CM$.
\item For every $x\in X$,
$\supp((\ev_x)_*\mu_K)\subseteq\overline{u\CM}\,x\subseteq\CM x$.
\item If $(\ev_x)_*\mu_K$ is $G$-invariant, then
$\supp((\ev_x)_*\mu_K)=\CM x$ is a minimal flow.
\end{enumerate}
\end{remark}

\begin{proof}
(1) If not, then there is $F \in C(E)_+$ and $\epsilon >0$ such that  $\mu_K(\{\eta \in E: F(\eta)>\epsilon\}) >0$ and $F |_{\cl(u\mathcal{M})} =0$.  Then $0<\int_E F d\mu_K = \int_K F(j(k)) d \mathfrak{h}_K(k) = 0$, a contradiction.

(2) If $\mu_K$ is $G$-invariant, then so is its support which is also closed, so the conclusion follows from (1) by minimality of $\mathcal{M}$.
Also (3) follows from (1).

(4) By $G$-invariance,  $\supp ((\ev_x)_* \mu_K)$ is a subflow of $(X,G)$ which, by (3), is contained in the minimal subflow $\mathcal{M}x$, and so it must be equal to $\mathcal{M}x$.
\end{proof}

For every $x\in X$, $\supp\bigl((\ev_x)_*\mu_K\bigr)\subseteq\CM x$,
and $\CM x$ is a minimal subflow of $X$. Consequently, whenever
$(\ev_x)_*\mu_K$ is $G$-invariant, its support is minimal.
This already
shows that evaluation of Haar decompressions cannot produce invariant
measures with nonminimal support.
Proposition~\ref{prop:tame-ergodic-support-minimal} below shows that
an ergodic invariant measure with nonminimal support cannot occur on a tame flow.
For comparison, an ergodic Bernoulli measure on a full shift has
nonminimal support; however, the full shift is not tame.

An interesting feature of decompression is that $\mu_K$ may become
$G$-invariant on $E(X,G)$ even when invariance is not visible on the
original Ellis group. The split-circle dihedral example in
Appendix~\ref{subsec:split-circle-ordinary} exhibits this phenomenon.

\section{Unique ergodicity and minimal supports in tame flows}
\label{sec:unique.minimal}
Tameness imposes strong rigidity on invariant measures. We prove
(without metrizability assumptions) that every minimal tame amenable flow
is uniquely ergodic and that every ergodic measure
on a tame flow has minimal support. 
Consequently, the ergodic measures are canonically parametrized by the amenable minimal subflows.

\subsection{Unique ergodicity}
We begin with the metrizable case, where the almost one-to-one
description of minimal tame flows identifies the evaluation of every
Haar decompression with the unique invariant probability measure.
Then, we remove metrizability and obtain unique
ergodicity for arbitrary minimal tame amenable flows.

We first recall the relevant terminology.
Let $\pi:X\to Z$ be a factor map of compact metrizable flows, and put
$$Z_0:=\{z\in Z:|\pi^{-1}(z)|=1\}.$$
The extension $\pi$ is \emph{almost one-to-one} if $Z_0$ is dense in
$Z$. If $(Z,G)$ is uniquely ergodic, with invariant probability measure
$\lambda_Z$, then $\pi$ is called \emph{regular} if
$\lambda_Z(Z_0)=1$.

\begin{lemma}\label{thm:ordinary-ellis-glasner-regime}
Let $(X,G)$ be a metrizable minimal tame amenable flow, let $x_0\in X$, and let $\CM\unlhdm E:=E(X,G)$, $u^2=u\in\CM$ and $K:=u\CM$.
Then $\nu_K:=(\ev_{x_0})_*\mu_K$ is a unique $G$-invariant regular Borel probability measure on $X$.
\end{lemma}

\begin{proof}
Let 
$$\pi:X\longrightarrow Z,\qquad z_0:=\pi(x_0),$$
be the maximal equicontinuous factor of $(X,G)$. By
\cite[Corollary~5.4(2)]{Glasner2}, the map $\pi$ is almost one-to-one and $X$ is
uniquely ergodic. Moreover, by \cite[Theorem~1.2]{FGJO21}, the extension is regular. 
The minimal equicontinuous flow $(Z,G)$ is uniquely ergodic.
Thus, if $\lambda_Z$ denotes the unique invariant probability measure
on $Z$, then
\begin{equation}
\label{eq:singleton-fibres-full-measure}
        \lambda_Z(Z_0)=1,
        \qquad
        Z_0:=\{z\in Z:|\pi^{-1}(z)|=1\}.
\end{equation}

Put $L:=E(Z,G)$.
Since $(Z,G)$ is minimal and equicontinuous, $L$ is a compact Hausdorff
topological group, and
\begin{equation}
\label{eq:haar-on-mef}
        \lambda_Z=(\ev_{z_0})_*\mathfrak h_L,
\end{equation}
where $\mathfrak h_L$ is the normalized Haar measure on $L$.
The factor map $\pi$ induces a continuous surjective semigroup homomorphism
$\pi_E:E(X,G)\longrightarrow E(Z,G)=L$
such that $\pi\circ\eta=\pi_E(\eta)\circ\pi$
for every $\eta\in E$.
We have $\pi_E[\CM]=L$ and $\pi_E(u)$ is equal to the identity of
$L$, and the restriction $\pi_E|_K:K\longrightarrow L$
is a continuous epimorphism of compact Hausdorff groups.
Hence,
\begin{equation}
\label{eq:haar-factor-pushforward}
        (\pi_E|_K)_*\mathfrak h_K=\mathfrak h_L.
\end{equation}

We now compute the projection of $\nu_K$ to $Z$. 
For every
$\varphi\in C(Z,\mathbb C)$, the function
$$E\ni\eta\longmapsto\varphi\bigl(\pi(\eta(x_0))\bigr)\in\mathbb{C},$$
is continuous. Using decompression,
\eqref{eq:haar-factor-pushforward}, and \eqref{eq:haar-on-mef}, we obtain
\begin{align*}
        \int_Z \varphi\,d(\pi_*\nu_K)
        &=
        \int_E \varphi\bigl(\pi(\eta(x_0))\bigr)
        \,d\mu_K(\eta)
        =
        \int_K \varphi\bigl(\pi(k(x_0))\bigr)
        \,d\mathfrak h_K(k)\\
        &=
        \int_K \varphi\bigl(\pi_E|_K(k)z_0\bigr)
        \,d\mathfrak h_K(k)
        =
        \int_L \varphi(\ell z_0)\,d\mathfrak h_L(\ell)\\
        &=
        \int_Z \varphi\,d\lambda_Z.
\end{align*}
Hence,
\begin{equation}
\label{eq:ordinary-decompression-projects-to-mef-haar}
        \pi_*\nu_K=\lambda_Z.
\end{equation}

It remains to use regularity of the almost one-to-one extension.
By a standard argument the set $Z_0$
is a $G_\delta$ subset of $Z$, and hence is Polish. The set
$X_0:=\pi^{-1}(Z_0)$ is likewise a $G_\delta$ subset of $X$, and hence is Polish. The restriction
$$
\pi|_{X_0}:X_0\longrightarrow Z_0
$$
is a continuous bijection. By the Lusin--Souslin theorem, its inverse
$$
s:Z_0\longrightarrow X_0
$$
is Borel.
Now, let $\rho$ be a Borel probability measure on $X$ such that
$\pi_*\rho=\lambda_Z$. Since $\lambda_Z(Z_0)=1$, we have $\rho(X_0)=1$.
Therefore, for every Borel set $B\subseteq X$,
$$
\rho(B)
=
\lambda_Z\bigl(s^{-1}(B\cap X_0)\bigr).
$$
Thus, $\rho$ is uniquely determined by $\lambda_Z$, and consequently there
is exactly one Borel probability measure on $X$ which projects to
$\lambda_Z$.

By \eqref{eq:ordinary-decompression-projects-to-mef-haar}, the measure
$\nu_K$ is such a lift. The unique invariant measure, say $\mu_X$, is also
such a lift, because $\pi_*\mu_X$ is $G$-invariant on the uniquely ergodic
flow $(Z,G)$ and hence equals $\lambda_Z$.
\end{proof}

\begin{theorem}\label{thm:nonmetric-unique-ergodicity}
Every minimal tame amenable flow is uniquely ergodic.
\end{theorem}

\begin{proof}
Let $(Y,G)$ be minimal and tame, and suppose that
$\mu_0,\mu_1\in\operatorname{Prob}_G(Y)$ are distinct.  Choose
$f\in C(Y)$ such that
$$\int_Y f\,d\mu_0\neq\int_Y f\,d\mu_1.$$
We inductively construct a countable subgroup $G_0\leqslant G$ and a metrizable
minimal $G_0$-factor through which $f$ descends.

Put $A_0=C^*(1,f)\subseteq C(Y)$ and
$G_{0,0}=\{1\}$.  Now, suppose that 
a countable subgroup $G_{0,n}\leqslant G$
and a separable unital
$G_{0,n}$-invariant $C^*$-algebra $A_n\subseteq C(Y)$ have been chosen.  
Let
$Y_n=\operatorname{Spec}(A_n)$, let $\pi_n:Y\to Y_n$ be the Gelfand
dual of $A_n\subseteq C(Y)$, i.e. $\pi_n(y)(a):=a(y)$,
and choose a countable base $\mathcal U_n$ for $Y_n$.  For
each nonempty $U\in\mathcal U_n$, minimality and compactness give a
finite set $F_U\subseteq G$ such that
$$Y=\bigcup_{g\in F_U}g\pi_n^{-1}[U].$$
Let $G_{0,n+1}$ be generated by $G_{0,n}$ together with all $F_U$ for 
$\emptyset\neq U\in\mathcal U_n$. Let $A_{n+1}:=C^*\bigl(G_{0,n+1}\cdot A_n\bigr)$.

Set
$$G_0=\bigcup_{n<\omega}G_{0,n},\qquad
 A=\overline{\bigcup_{n<\omega} A_n},\qquad Z=\operatorname{Spec}(A),$$
and denote the Gelfand dual of $A\subseteq C(Y)$ by $\pi:Y\to Z$.
For $y\in Y$ and  $a\in A$, we have $\pi(y)(a)=a(y)$.
The algebra $A$ is
separable and $G_0$-invariant; hence $Z$ is a metrizable
$G_0$-flow. 

For each $n<\omega$, let $r_n:Z\to Y_n$ be the restriction map
induced by $A_n\subseteq A$; thus $\pi_n=r_n\circ\pi$.  If $n\leqslant m$,
let $r_{m,n}:Y_m\to Y_n$ be the restriction map induced by
$A_n\subseteq A_m$.  Then $r_n=r_{m,n}\circ r_m$.
The compatible family $(r_n)_{n<\omega}$ induces a continuous map
$$R:Z\longrightarrow\prod_{n<\omega}Y_n,
 \qquad R(z)=(r_n(z))_{n<\omega}.$$
This map is injective.  Indeed, if $R(z)=R(z')$, then the characters
$z$ and $z'$ agree on every $A_n$, hence on the norm-dense subalgebra
$\bigcup_nA_n$ of $A$, and therefore on all of $A$.  Since $Z$ is
compact and $\prod_nY_n$ is Hausdorff, $R$ is a topological
embedding.  Thus the topology of $Z$ is the initial topology induced
by the family $(r_n)_{n<\omega}$.

Consequently, if
$\emptyset \neq V\subseteq Z$ is open, there are $n<\omega$ 
and a nonempty
$U\in\mathcal U_n$ such that $r_n^{-1}[U]\subseteq V$.  For
$g\in F_U\subseteq G_0$, equivariance of $\pi$ gives
$$\pi^{-1}\bigl[g r_n^{-1}[U]\bigr]=g\pi_n^{-1}[U].$$
We have
$$\pi^{-1}\Big[\bigcup_{g\in F_U}g r_n^{-1}[U]\Big]=
 \bigcup_{g\in F_U}\pi^{-1}\bigl[g r_n^{-1}[U]\bigr] =
 \bigcup_{g\in F_U}g\pi_n^{-1}[U]=Y.$$
Since $\pi:Y\to Z$ is surjective, it follows that
$Z=\bigcup_{g\in F_U}g r_n^{-1}[U]$, and hence
$$Z=\bigcup_{g\in F_U}g r_n^{-1}[U]\subseteq\bigcup_{g\in F_U}gV.$$
Thus $(Z,G_0)$ is minimal.

Since $f\in A$, its Gelfand transform $\hat{f}:Z\to\mathbb C$, 
($\hat{f}(z)=z(f)$ for $z\in Z$) belongs to $C(Z)$.  
Moreover, for every $y\in Y$, $\hat{f}(\pi(y))=\pi(y)(f)=f(y)$.
Thus $f=\hat{f}\circ\pi$.

Note that $(Z,G_0)$ is tame, 
$\pi_*\mu_0$ and $\pi_*\mu_1$ are $G_0$-invariant
and are separated by $\hat{f}$:
$$\int \hat{f}d \pi_* \mu_0=\int (\hat{f} \circ \pi)d \mu_0 = \int f d \mu_0  \ne \int fd \mu_1=\int \hat{f}d \pi_* \mu_1.$$
The metrizable minimal tame flow
$(Z,G_0)$ is therefore amenable but not uniquely ergodic.  This
contradicts Lemma \ref{thm:ordinary-ellis-glasner-regime}.
\end{proof}

\subsection{Minimality of supports}
We start with a very useful lemma which is our technique to replace the Lebesgue dominated convergence theorem when working with general nets. 
It will be crucial in Subsection \ref{subsec:Haar.invariance}, and here it shortens a bit the proof of Proposition \ref{prop:tame-ergodic-support-minimal}.
More precisley, let $\eta\in E(Y,G)$ be a limit of a net $(g_i)_{i\in I}$ in $G$, let $\mu\in\Prob_G(Y)$ and $f\in C(Y)$. For each $i\in I$ we have
$$\int_Y f\circ g_i\,d\mu= \int_Y f\,d(g_i)_*\mu=\int_Y f\,d\mu,$$
but it is not clear whether we can pass with the above equalities to the limit of $(g_i)_{i\in I}=\eta$, to obtain $\int f\circ\eta \,d\mu=\int f\,d\mu$. Tameness helps as follows:

\begin{lemma}\label{lemma:tameness.dominated.convergence}
    Let $(Y,G)$ be tame, $\nu\in\Prob(Y)$, $f\in C(Y)$. 
    Let $\eta\in E(Y,G)$ and let $\eta^\#\in E(\Prob(Y),G)$ be as in Fact \ref{fact:tame-injective}. Then, we have
    $$\int_Y f\circ\eta \,d\nu = \int_Y f\,d\eta^\# (\nu).$$
\end{lemma}

\begin{proof}
Without loss of generality we may assume that $f\in C(Y,\mathbb{R})$. 
By Lemma \ref{lem:finite-powers-of-tame-flows}(3), Fact \ref{fact:tame-injective}, and preservation of tameness under taking factors,
$(\Prob(Y),G)$ is tame, and thus 
$$\eta^\#:\Prob(Y)\to\Prob(Y)$$
is a fragmented map. The map $\hat{f}:\Prob(Y)\to\mathbb{R}$, given by $\hat{f}(\lambda)=\int f\,d\lambda$, is continuous. Thus, the map $\hat{f}\circ\eta^\#:\Prob(Y)\to \mathbb{R}$ is fragmented, bounded and affine.
Thus, by \cite[Theorem 4.21]{LMNS10}, $\hat{f}\circ\eta^\#$ satisfies the barycentric formula (cf. the last paragraph of Subsection \ref{subsec:flows-and-invariant-measures}).

Let $\delta:Y\to\Prob(Y)$ denote the embedding via Dirac measures, i.e. $\delta(y)=\delta_y$. Consider the measure $\delta_*\nu\in\Prob\big(\Prob(Y)\big)$.

\begin{clm}
    The barycenter $b(\delta_*\nu)\in\Prob(Y)$ 
    of $\delta_*\nu$ is equal to $\nu$.
\end{clm}

\begin{clmproof}
    Let $h\in C(Y,\mathbb{R})$ and consider $\hat{h}:\Prob(Y)\to\mathbb{R}$,
    given by  $\hat{h}(\lambda):=\int h d\lambda$. 
    Applying the barycentric formula for $\hat{h}$, we have
    \begin{IEEEeqnarray*}{rCcCl}
        \hat{h}\big(b(\delta_*\nu)\big) &=& \int\limits_{\mu\in\Prob(Y)} \hat{h}(\mu)\,d(\delta_*\nu)(\mu) 
        &=& \int\limits_Y \hat{h}(\delta_y)\,d\nu(y) \\
        &=& \int\limits_Y h(y)\,d \nu(y) = \hat{h}(\nu) &&
    \end{IEEEeqnarray*}
    Because $h\in C(Y,\mathbb{R})$ was arbitrary, we conclude $b(\delta_*\nu)=\nu$.
\end{clmproof}
Now, we use Claim 1 and apply the barycentric formula to $\hat{f}\circ\eta^\#$:
\begin{IEEEeqnarray*}{rCl}
    \int\limits_Yf\, d(\eta^\#\nu) &=& (\hat{f}\circ\eta^\#)(\nu) = 
    (\hat{f}\circ\eta^\#)\big(b(\delta_*\nu)\big) \\
    &=& \int\limits_{\mu\in\Prob(Y)}(\hat{f}\circ\eta^\#)(\mu)\,d(\delta_*\nu)(\mu) \\
    &=& \int\limits_Y (\hat{f}\circ\eta^\#\circ\delta)(y)\,d\nu(y) \\
    &=& \int\limits_Y \hat{f}(\delta_{\eta(y)})\,d\nu(y)=\int\limits_Y f\circ\eta\,d\nu
\end{IEEEeqnarray*}
\end{proof}

A special instance of the minimal-support phenomenon below was obtained
for tame semicascades, under additional hypotheses, in
\cite[Theorem~3.6]{Romanov16}. We establish the result here for arbitrary
tame flows, without metrizability assumptions and for general acting
groups.

\begin{proposition}\label{prop:tame-ergodic-support-minimal}
Let $(X,G)$ be a tame flow and let $\mu\in\Erg_G(X)$. Then
$\supp(\mu)$ is a minimal subflow of $X$.
\end{proposition}

\begin{proof}
Set $Y:=\supp(\mu)$, and regard $\mu$ as a probability measure on $Y$.
Then $(Y,G)$ is tame, $\mu\in\Erg_G(Y)$, and $\mu$ has full support.

Suppose, towards a contradiction, that $Y$ is not minimal. Choose a
proper nonempty subflow $C\subsetneq Y$, a point $y_0\in Y\setminus C$,
and a function $f\in C(Y,[0,1])$ such that
$f|_C=0$ and $f(y_0)=1$.
As in the proof of Theorem~\ref{thm:nonmetric-unique-ergodicity}, we
shall construct a countable subgroup $G_0\leqslant G$ and a metrizable
tame $G_0$-factor through which $f$ descends. The construction will also
ensure that every nonempty open subset of the factor has conull
$G_0$-orbit with respect to the projected measure.

Put $A_0:=C^*(1,f)\subseteq C(Y)$, $G_{0,0}:=\{1_G\}$.
Assume that a countable subgroup $G_{0,n}\leqslant G$ and a separable
unital $G_{0,n}$-invariant $C^*$-algebra $A_n\subseteq C(Y)$ have been chosen. 
Let
$$Y_n:=\operatorname{Spec}(A_n),$$
let $\pi_n:Y\to Y_n$ be the Gelfand dual of $A_n\subseteq C(Y)$, and
choose a countable base $\mathcal U_n$ for $Y_n$.
Fix a nonempty $U\in\mathcal U_n$ and put $W_U:=\pi_n^{-1}[U]$.
The set $W_U$ is nonempty and open, so $\mu(W_U)>0$. Consider an open and $G$-invariant set
$$GW_U=\bigcup_{g\in G}gW_U.$$
By \cite[Proposition~12.4]{Phelps2001}, ergodicity of $\mu$ implies $\mu(GW_U)=1$. By regularity and compactness, there is a countable set $F_U\subseteq G$ such that
$$\mu\big(\bigcup_{g\in F_U}gW_U\big)=1.$$
Indeed, for every $m\geqslant1$, choose a compact set $C_m\subseteq GW_U$ 
with $\mu(C_m)>1-1/m$ and then a finite set
$F_{U,m}\subseteq G$ such that
$$C_m\subseteq\bigcup_{g\in F_{U,m}}gW_U.$$
We set $F_U:=\bigcup_{m\geqslant1}F_{U,m}$.

Let $G_{0,n+1}$ be generated by $G_{0,n}$ together with all the sets
$F_U$, where $\emptyset\neq U\in\mathcal U_n$, and let $A_{n+1}:=C^*\bigl(G_{0,n+1}\cdot A_n\bigr)$.
Set
$$G_0:=\bigcup_{n<\omega}G_{0,n},\qquad
A:=\overline{\bigcup_{n<\omega}A_n},\qquad
Z:=\operatorname{Spec}(A),$$
and let $\pi:Y\to Z$ be the Gelfand dual of $A\subseteq C(Y)$. 
The flow $(Z,G_0)$ is tame and metrizable. We will show that it is also a minimal flow.

For each $n<\omega$, let $r_n:Z\longrightarrow Y_n$
be the restriction map induced by $A_n\subseteq A$; thus
$\pi_n=r_n\circ\pi$. If $n\leqslant m$, let
$r_{m,n}:Y_m\to Y_n$ be induced by $A_n\subseteq A_m$. Then
$r_n=r_{m,n}\circ r_m$. Exactly as in the proof of
Theorem~\ref{thm:nonmetric-unique-ergodicity}, the map
$$R:Z\longrightarrow\prod_{n<\omega}Y_n,
\qquad R(z):=(r_n(z))_{n<\omega},$$
is a topological embedding. Consequently, if
$\emptyset\neq V\subseteq Z$ is open, then there are $n<\omega$ and a
nonempty $U\in\mathcal U_n$ such that
$$r_n^{-1}[U]\subseteq V.$$

Put $\nu:=\pi_*\mu$. For every nonempty $U\in\mathcal U_n$ and every
$g\in F_U\subseteq G_0$, $G_0$-equivariance of $\pi$ gives
$\pi^{-1}\bigl[g r_n^{-1}[U]\bigr]=g\pi_n^{-1}[U]$.
It follows from the choice of $F_U$ that
\begin{IEEEeqnarray*}{rCl}
    \nu\big(\bigcup_{g\in F_U}g r_n^{-1}[U]\big) &=& 
    \mu\Big(\pi^{-1}\big[\bigcup_{g\in F_U}g r_n^{-1}[U] \big]\Big)\\
    &=& \mu\big(\bigcup_{g\in F_U}g\pi_n^{-1}[U]\big)=1.
\end{IEEEeqnarray*}
Therefore, for every nonempty open subset $V\subseteq Z$, $\nu(G_0V)=1$.

Choose a countable base $(V_m)_{m<\omega}$ of nonempty open subsets of
$Z$ and put
$$T:=\bigcap_{m<\omega}G_0V_m.$$
Then $\nu(T)=1$, and every point of $T$ has dense $G_0$-orbit in $Z$.

We now use Lemma~\ref{lemma:tameness.dominated.convergence} to complete
the minimality argument. Let $E_Z:=E(Z,G_0)$, choose a minimal left ideal
$\mathcal M_Z\unlhd_m E_Z$, and let $u\in\mathcal M_Z$ be an idempotent.
Since $\nu$ is $G_0$-invariant, it is fixed by every element of
$E(\Prob(Z),G_0)$. In particular, $u^\#(\nu)=\nu$ (notation as in Fact \ref{fact:tame-injective}).
Since $Z$ is compact metrizable and $u$ is fragmented, $u$ is Borel
measurable (by \cite[Lemma~2.5(2)]{GlasnerMegrelishvili2012}, $u$ is of
Baire class one). 
Because $Z$ is compact metrizable, $u_*\nu$ is regular \cite[Theorem~17.10]{Kechris95}, and in consequence,
since by Lemma \ref{lemma:tameness.dominated.convergence} for every $f \in C(Z)$ we have
$$\int_Z fd u^{\#}\nu = \int_Z (f \circ u)d\nu=\int_Z f d u_*\nu,$$ 
we get $u^{\#} \nu = u_* \nu$.
In particular, for every Borel set $A\subseteq Z$,
$$\bigl(u^\#(\nu)\bigr)(A)=\nu\bigl(u^{-1}[A]\bigr).$$

Let $F:=\{z\in Z:uz=z\}$.
The set $F$ is Borel, and $u^2=u$ implies that $u^{-1}[F]=Z$. Therefore,
$$\nu(F)=\bigl(u^\#(\nu)\bigr)(F)=\nu\bigl(u^{-1}[F]\bigr)=1.$$
Choose $z\in F\cap T$. Then $uz=z$ and
$\overline{G_0z}=Z$, so
$$Z=E_Zz=E_Zuz=(E_Zu)z=\mathcal M_Zz.$$
Thus, $(Z,G_0)$ is minimal.

Finally, since $f\in A$, its Gelfand transform $\widehat f\in C(Z)$
satisfies $f=\widehat f\circ\pi$.
The set $\pi(C)$ is a nonempty closed $G_0$-invariant subset of the
minimal flow $Z$, and hence $\pi(C)=Z$. Since $f|_C=0$, it follows that
$\widehat f=0$ (on $Z$), contradicting $f(y_0)=1$. Therefore,
$Y=\supp(\mu)$ is minimal.
\end{proof}

\begin{remark}\label{rem:ergodicity.notnions}
Let $(X,G)$ be a flow. 
Recall that a $G$-invariant measure $\mu$ is ergodic if it is an extreme point of
$\Prob_G(X)$. Alternatively (by \cite[Proposition 12.4]{Phelps2001}), 
if for every measurable $A\subseteq X$
such that $\mu(gA\,\triangle\, A)=0$ holds for each $g\in G$,
we have that $\mu(A)\in\{0,1\}$. In the literature, there is a weaker notion of ergodicity:
\begin{enumerate}
    \item[$(\ddagger)$]  $\mu$ is $G$-invariant and for every measurable $G$-invariant $A\subseteq X$ we have $\mu(A)\in\{0,1\}$.
\end{enumerate}
We claim that $(\ddagger)$ is equivalent to ergodicity for tame flows:
Let $(X,G)$ be tame and let $\mu$ satisfy condition $(\ddagger)$. Set $Y:=\supp(\mu)$ and note that the proof of Proposition \ref{prop:tame-ergodic-support-minimal} might be repeated using only condition $(\ddagger)$. Thus $(Y,G)$ is tame, minimal and amenable. By Theorem \ref{thm:nonmetric-unique-ergodicity},
$\mu\in\Erg_G(Y)$. Now, if $\nu_1,\nu_2\in\Prob_G(X)$ and $\mu=a_1\nu_1+a_2\nu_2$
for some $0<a_1,a_2$ with $a_1+a_2=1$, then $\supp(\nu_1),\supp(\nu_2)\subseteq Y$.
Then, by Theorem \ref{thm:nonmetric-unique-ergodicity}, $\nu_1=\mu=\nu_2$.
\end{remark}

\begin{cor}\label{cor:ergodic-measures-minimal-subflows}
Let $(X,G)$ be tame, and let
$\operatorname{Min}^{\mathrm{am}}_G(X)$ denote the collection of all
amenable minimal subflows of $X$. Then the support map
$$
\Erg_G(X)\longrightarrow\operatorname{Min}^{\mathrm{am}}_G(X),
\qquad
\mu\longmapsto\supp(\mu),
$$
is a bijection. Its inverse assigns to each
$Y\in\operatorname{Min}^{\mathrm{am}}_G(X)$ its unique invariant
probability measure $\mu_Y$, regarded as a measure on $X$.

Consequently, distinct ergodic invariant probability measures on $X$
have disjoint supports. If $(X,G)$ is hereditarily amenable, then
$\operatorname{Min}^{\mathrm{am}}_G(X)$ is the collection of all minimal
subflows of $X$.
\end{cor}

\begin{proof}
Let $\mu\in\Erg_G(X)$. By
Proposition~\ref{prop:tame-ergodic-support-minimal},
$\supp(\mu)$ is minimal. Thus,
$\supp(\mu)\in\operatorname{Min}^{\mathrm{am}}_G(X)$.

Conversely, let $Y\in\operatorname{Min}^{\mathrm{am}}_G(X)$. Since
$(Y,G)$ is minimal, tame and amenable,
Theorem~\ref{thm:nonmetric-unique-ergodicity} gives a unique invariant
probability measure $\mu_Y$ on $Y$. We have $\supp(\mu_Y)=Y$.
Moreover, $\mu_Y$ is ergodic when regarded as a measure on $X$.
Indeed, if
$$\mu_Y=t\nu_0+(1-t)\nu_1,\qquad 0<t<1,\qquad \nu_0,\nu_1\in\Prob_G(X),$$
then concentration of $\mu_Y$ on $Y$ forces both $\nu_0$ and $\nu_1$
to be concentrated on $Y$. The uniqueness of the invariant probability
measure on $Y$ therefore gives
$\nu_0=\nu_1=\mu_Y$.
\end{proof}

\section{Invariant measures on the Ellis semigroup}
\label{sec:invariance}
We first
characterize amenability of the Ellis flow in terms of hereditary
amenability of all finite powers of $X$; this criterion does not require
tameness (Theorem~\ref{thm:ellis-flow-amenable}). We then prove that, for
tame flows, ordinary hereditary amenability implies the above finite-power
condition
(Theorem~\ref{thm:hereditary-amenability-finite-powers}), which yields
Corollary~\ref{cor:ellis-amenability-characterization}.
Finally,
Theorem~\ref{thm:haar-decompression-amenability} shows that the Haar decompression
is $G$-invariant whenever the corresponding minimal left ideal $\CM$ is amenable as a $G$-flow.
(We do not assume that $(X,G)$ has a dense orbit.)

\subsection{Finite-power hereditary amenability}
Theorem~\ref{thm:ellis-flow-amenable} below shows, without any
tameness assumption, that amenability of the Ellis flow is equivalent
to hereditary amenability of all finite powers of $X$. Its proof uses
the dense algebra of finite-coordinate functions introduced below.
This motivates the following definition.

\begin{definition}
\label{def:her.amenability}
The flow $X$ is \emph{finitely power-hereditarily amenable} if, for every
$n\geqslant1$, every subflow of $X^n$, equipped with the diagonal action
$$
g(x_1,\ldots,x_n):=(gx_1,\ldots,gx_n),
$$
carries a $G$-invariant regular Borel probability measure.
\end{definition}

We shall use the algebra of finite-coordinate functions
$$\mathcal A_{\mathrm{fin}}:=\left\{
F\in C(E(X,G),\mathbb C):\begin{array}{l}
\text{for some $n\geqslant1$, $\bar x\in X^n$, and
$\varphi\in C(X^n,\mathbb C)$,}\\ F(\eta)=\varphi(\eta\bar x)
\text{ for every $\eta\in E(X,G)$}\end{array}\right\}.$$
By concatenating the tuples on which two functions depend, one sees that $\mathcal A_{\mathrm{fin}}$ is closed under sums and products; it is also closed under complex conjugation, scalar multiplication
and contains the constants. It separates points of $E(X,G)$: if $\eta\neq\xi$, choose $x\in X$ with
$\eta(x)\neq\xi(x)$ and then choose $f\in C(X,\mathbb C)$ separating these two points. The function
$$\zeta\longmapsto f(\zeta(x))$$
belongs to $\mathcal A_{\mathrm{fin}}$ and separates $\eta$ and $\xi$.
Hence, by the complex Stone--Weierstrass theorem,
$$
\overline{\mathcal A_{\mathrm{fin}}}^{\|\cdot\|_\infty}
=
C(E(X,G),\mathbb C).
$$

\begin{remark}
\label{rem:ellis-amenability-is-hereditary}
$(E(X,G),G)$ is amenable if and only if $(E(X,G),G)$ is hereditarily amenable.
\end{remark}

\begin{proof}
Only the left-to-right implication requires proof. 
Put $E:=E(X,G)$, and let $\nu$ be a $G$-invariant regular Borel probability measure on $E$.

Let $Y\subseteq E$ be a nonempty subflow and fix $q\in Y$.
By the left-continuity in $E$, the right translation
$$R_q:E\longrightarrow E,\qquad R_q(p):=pq,$$
is continuous and $G$-equivariant (i.e. $R_q(gp)=gR_q(p)$). Moreover, $R_q[E]=Eq=\overline{Gq}$.
Since $Y$ is closed and $G$-invariant and $q\in Y$, one has
$\overline{Gq}\subseteq Y$. Thus, $(R_q)_*\nu$ is a probability measure
on $Y$. Because $R_q$ is $G$-equivariant, $(R_q)_*\nu$ is $G$-invariant.
\end{proof}

\begin{theorem}
\label{thm:ellis-flow-amenable}
$\bigl(E(X,G),G\bigr)$ is amenable if and only if
$(X,G)$ is
finitely power-hereditarily amenable.
\end{theorem}

\begin{proof}
Assume first that $\bigl(E(X,G),G\bigr)$ is amenable.
Fix \(n\geqslant 1\), and let \(W\subseteq X^n\) be a nonempty subflow. Choose a
minimal subflow \(V\subseteq W\) and a point $\bar x=(x_1,\ldots,x_n)\in V$.
Consider the evaluation map
$$
e_{\bar x}:E(X,G)\longrightarrow V,
\qquad
e_{\bar x}(\eta):=\eta(\bar x).
$$
It is continuous and \(G\)-equivariant.
Moreover, the image of \(e_{\bar x}\) is compact and contains
\(G\bar x\), so minimality of \(V\) gives $e_{\bar x}[E(X,G)]=V$.

The pushforward of a \(G\)-invariant probability measure on \(E(X,G)\)
is therefore a \(G\)-invariant regular Borel probability measure on \(V\),
and hence, 
can be viewed as a measure on $W$ which is supported on $V$.

Now, we prove that if $(X,G)$ is
finitely power-hereditarily amenable then the flow
$\bigl(E(X,G),G\bigr)$ is amenable.
Let $E:=E(X,G)$.
For $g\in G$, $F\in C(E,\mathbb C)$, and $\varepsilon>0$, consider
$$D(g,F,\varepsilon):=\left\{\rho\in\Prob(E):\left|\int_E F(g\eta)\,d\rho(\eta)-\int_E F(\eta)\,d\rho(\eta)\right|\leqslant\varepsilon\right\}.$$
Each $D(g,F,\varepsilon)$ is weak-$*$ closed in $\Prob(E)$.
As $\Prob(E)$ is weak-$*$ compact, it is enough to show that
these sets have
the finite-intersection property (see the end of the proof).

Fix $g_1,\ldots,g_r\in G$,
$F_1,\ldots,F_r\in C(E,\mathbb{C})$
and $\epsilon_1,\ldots,\epsilon_r>0$.
By the density of the algebra
$\mathcal{A}_{\mathrm{fin}}$ in $C(E,\mathbb C)$ (see the paragraph after Definition~\ref{def:her.amenability}), for
every $j\leqslant r$ choose $A_j\in \mathcal{A}_{\mathrm{fin}}$
such that
\begin{equation}
\label{eq:ellis-amenable-approximation}
\|F_j-A_j\|_\infty<\frac{\varepsilon_j}{3}.
\end{equation}
After concatenating the finitely many tuples on which the functions $A_j$
depend, there are a single finite tuple
$\bar x=(x_1,\ldots,x_n)\in X^n$  and functions $\varphi_j\in C(X^n,\mathbb C)$ such that
\begin{equation}
\label{eq:ellis-amenable-coordinate-form}
A_j(\eta)=\varphi_j(\eta\bar x)
\qquad
(\eta\in E,\ 1\leqslant j\leqslant r).
\end{equation}

Consider $e_{\bar x}:E\longrightarrow X^n$ given by $e_{\bar x}(\eta):=\eta(\bar x)$, and set $Y_{\bar x}:=e_{\bar x}[E]$.
The set $Y_{\bar x}$ is a subflow of $X^n$, and $Y_{\bar x}=\overline{G\bar x}$.

Since $X$ is finitely power-hereditarily amenable, there is a
$G$-invariant regular Borel probability measure
$\lambda$ on $Y_{\bar x}$.

\begin{clm}
$\lambda$ has a lift $\rho$ to $E$.    
\end{clm}

\begin{clmproof}
Since
$e_{\bar x}:E\to Y_{\bar x}$ is a continuous surjection, the formula
$$L_0(f\circ e_{\bar x}):=\int_{Y_{\bar x}}f\,d\lambda$$
defines a state on the unital $C^\ast$-subalgebra
$e_{\bar x}^\ast C(Y_{\bar x},\mathbb C)
\subseteq
C(E,\mathbb C)$.
By Krein extension theorem, we can extend $L_0$ to
a state on $C(E,\mathbb C)$. By the Riesz--Markov theorem,
the extension is integration against some
regular Borel probability measure $\rho$ on $E$,
and, by construction,
\begin{equation}
\label{eq:ellis-amenable-lift}
(e_{\bar x})_*\rho=\lambda.
\end{equation}
\end{clmproof}

\noindent
Using \eqref{eq:ellis-amenable-coordinate-form}, \eqref{eq:ellis-amenable-lift}, and the $G$-invariance of $\lambda$, we
obtain, for every $1\leqslant j\leqslant r$,
\begin{IEEEeqnarray*}{rCcCl}
\int_E A_j(g_j\eta)\,d\rho(\eta) &=& \int_{Y_{\bar x}}\varphi_j(g_jy)\,d\lambda(y)
&=& \int_{Y_{\bar x}}\varphi_j(y)\,d\lambda(y)\\
&=& \int_E A_j(\eta)\,d\rho(\eta). & &
\end{IEEEeqnarray*}
Consequently, by
\eqref{eq:ellis-amenable-approximation},
\begin{align*}
&
\left|
\int_E F_j(g_j\eta)\,d\rho(\eta)-\int_E F_j(\eta)\,d\rho(\eta)
\right|
\\
&\qquad\leqslant
\left|
\int_E\bigl(F_j-A_j\bigr)(g_j\eta)\,d\rho(\eta)
\right|+\left|
\int_E\bigl(F_j-A_j\bigr)(\eta)\,d\rho(\eta)
\right|
\\
&\qquad\leqslant
2\|F_j-A_j\|_\infty<\frac{2\varepsilon_j}{3}<\varepsilon_j.
\end{align*}
Thus,
$$\rho\in\bigcap_{j=1}^rD(g_j,F_j,\varepsilon_j),$$
which proves the finite-intersection property.

By compactness of $\Prob(E)$, there is
$$\rho_0\in
\bigcap_{\substack{g\in G,\ F\in C(E,\mathbb C)\\ m\geqslant1}}
D\left(g,F,\frac1m\right).$$
Therefore, for every $g\in G$ and every $F\in C(E,\mathbb C)$,
$$\int_E F(g\eta)\,d\rho_0(\eta)=\int_E F(\eta)\,d\rho_0(\eta).$$
Equivalently, $g_\ast\rho_0=\rho_0$ for every $g\in G$.
Hence, the Ellis flow $(E(X,G),G)$ carries a $G$-invariant regular Borel
probability measure.
\end{proof}

\subsection{Hereditary amenability for tame flows}
Tameness will be needed in order to pass
from hereditary amenability of $X$ to hereditary amenability of its finite powers. We first prove a joining lemma for minimal metrizable tame flows and then use a separable reduction to treat arbitrary compact tame flows.

\begin{lemma}\label{lem:old_claim1}
Let $H$ be a group, let
$Y_1,\ldots,Y_n$ be minimal metrizable tame $H$-flows, each carrying an
$H$-invariant regular Borel probability measure, and let
$$
V\subseteq Y_1\times\cdots\times Y_n
$$
be a minimal $H$-subflow. Then $V$ carries an $H$-invariant regular Borel
probability measure.
\end{lemma}

\begin{proof}
For each $i$, let
$$\pi_i:Y_i\longrightarrow Z_i$$
be the maximal equicontinuous factor. 
As in the proof of Lemma~\ref{thm:ordinary-ellis-glasner-regime},
by \cite[Corollary~5.4(2)]{Glasner2} and \cite[Theorem~1.2]{FGJO21}, if $\mathfrak m_i$ denotes
the unique invariant probability measure on $Z_i$, then
$$\mathfrak m_i(Z_i^0)=1, \qquad Z_i^0:=\left\{z\in Z_i:
|\pi_i^{-1}(z)|=1\right\}.$$
Define
$$\pi_V:V\longrightarrow Z_1\times\cdots\times Z_n$$
as $\pi_V:=\pi_1\times\ldots\times\pi_n$, and put $Z:=\pi_V[V]$.
By Lemma~\ref{lem:finite-powers-of-tame-flows}(1), the product
$Z_1\times\cdots\times Z_n$ is equicontinuous. Hence, its subflow $Z$
is equicontinuous. Since $Z$ is minimal, it has a unique invariant
probability measure $\mathfrak m_Z$.

For each $i$, the image of the coordinate projection
$V\longrightarrow Y_i$
is a nonempty closed $H$-invariant subset of the minimal flow $Y_i$,
thus equal to $Y_i$.
It follows that the induced projection $p_i:Z\longrightarrow Z_i$
is also onto. Therefore,
$$(p_i)_*\mathfrak m_Z=\mathfrak m_i.$$
Consequently, $\mathfrak m_Z(Z^0)=1$ for
$$Z^0:=Z\cap\bigl(Z_1^0\times\cdots\times Z_n^0\bigr)
=\bigcap_{i=1}^n p_i^{-1}[Z_i^0].$$

Now, if $z=(z_1,\ldots,z_n)\in Z^0$,
then
$$
\pi_V^{-1}(z)
\subseteq
\pi_1^{-1}(z_1)\times\cdots\times\pi_n^{-1}(z_n).
$$
The set on the right is a singleton. Since $z\in\pi_V[V]$, the fibre
$\pi_V^{-1}(z)$ is nonempty and hence is itself a singleton.

A probability lift of $\mathfrak m_Z$ to $V$ exists by applying the Krein extension theorem to extend the state 
$$f \circ \pi_V \mapsto \int_Z f\, d\mathfrak m_Z$$ 
from the unital $C^*$-subalgebra $\pi_V^* C(Z)$ to a state on $C(V)$, and subsequently applying the Riesz--Markov representation theorem.

This lift is unique. Indeed, each $Z_i^0$ is a $G_\delta$ subset of $Z_i$,
and hence $Z^0$ is a $G_\delta$ subset of the compact metrizable space
$Z$. Put $V^0:=\pi_V^{-1}[Z^0]$.
Then $V^0$ and $Z^0$ are Polish spaces, and
$\pi_V|_{V^0}:V^0\longrightarrow Z^0$
is a continuous bijection. By the Lusin--Souslin theorem, its inverse is
Borel. Every probability measure on $V$ which projects to
$\mathfrak m_Z$ is concentrated on $V^0$, and is therefore the pushforward
of $\mathfrak m_Z|_{Z^0}$ under this inverse. Thus, the lift is unique.

For every $h\in H$, the translate of this lift is again a lift of
$\mathfrak m_Z$. By uniqueness, the lift is $H$-invariant. 
\end{proof}

\begin{theorem}
\label{thm:hereditary-amenability-finite-powers}
Let $(X,G)$ be tame. If $X$ is hereditarily amenable, then
$X$ is finitely power-hereditarily amenable.
\end{theorem}

\begin{proof}
Fix $n\geqslant 1$, and let $W\subseteq X^n$ be a nonempty subflow. It is enough to
construct an invariant probability measure on a minimal subflow of $W$.
Replacing $W$ by such a subflow, we assume that $W$ is minimal.

For $1\leqslant i\leqslant n$, put
$$
Y_i:=\operatorname{pr}_i[W].
$$
Each $Y_i$ is a minimal subflow of $X$. By hereditary amenability, choose a
$G$-invariant regular Borel probability measure $\mu_i$ on $Y_i$.
For $g\in G$ and $f\in C(W,\mathbb C)$, we consider
$$C(g,f):=\left\{\nu\in\Prob(W):\int_W f(gw)\,d\nu(w)=
\int_W f(w)\,d\nu(w)\right\}.$$
Each $C(g,f)$ is weak-$\ast$ closed, and
$$\Prob_G(W)=\bigcap_{\substack{g\in G\\ f\in C(W,\mathbb C)}}C(g,f).$$
We verify that the above family has the finite-intersection property.

Fix $g_1,\ldots,g_r\in G$ and $f_1,\ldots,f_r\in C(W,\mathbb{C})$
and let $H_0\leqslant G$ be the subgroup generated by
$g_1,\ldots,g_r$. 
Since $W$ is a closed subset of the compact Hausdorff, and hence normal, space
$$Y_1\times\cdots\times Y_n,$$
the Tietze extension theorem, applied separately to the real and imaginary
parts, gives, for every $j\leqslant r$, an extension
$\widetilde f_j\in C(Y_1\times\cdots\times Y_n,\mathbb C)$ of $f_j$.

By the complex Stone--Weierstrass theorem, each $\widetilde f_j$ is a
uniform limit of finite sums of products of coordinate functions. Choose a
sequence of such approximations for every $j$, and, for
$1\leqslant i\leqslant n$, let $\mathcal D_i\subseteq C(Y_i,\mathbb C)$
be the countable family of all $i$-th coordinate functions which occur in these approximations.

We now construct recursively an increasing sequence of countable subgroups
$$H_0\leqslant H_1\leqslant H_2\leqslant\cdots\leqslant G$$
and, simultaneously for every $1\leqslant i\leqslant n$, an increasing
sequence of separable unital $H_m$-invariant $C^*$-subalgebras
$$A_{i,m}\subseteq C(Y_i,\mathbb C).$$
The action on functions is given by $(h\cdot a)(y):=a(h^{-1}y)$.
We already have $H_0$.
For each $i$, let $A_{i,0}$ be the unital $C^*$-algebra generated by
$\{h\cdot a:h\in H_0,\ a\in\mathcal D_i\}$.
Now, assume that $H_m$ and all the algebras $A_{i,m}$ have been constructed.
Let
$$q_{i,m}:Y_i\longrightarrow Y_{i,m}:=\operatorname{Spec}(A_{i,m})$$
be the associated factor maps of $H_m$-flows, with metrizable targets $Y_{i,m}$,
and put
$$q_m:=(q_{1,m},\ldots,q_{n,m})|_W,\qquad W_m:=q_m[W].$$
Choose a countable basis $\mathcal B_m$ consisting of nonempty open
subsets of $W_m$. For $B\in\mathcal B_m$, the set $U_B:=q_m^{-1}[B]$
is a nonempty open subset of the minimal $G$-flow $W$,
and so the family
$\{gU_B:g\in G\}$ covers $W$. By compactness, choose a finite set $F_B\subseteq G$ such that
$$W=\bigcup_{g\in F_B}gU_B.$$
Let $H_{m+1}$ be the subgroup generated by
$$H_m\cup\bigcup_{B\in\mathcal B_m}F_B.$$
This group is countable. For every $i$, let $A_{i,m+1}$ be the unital
$C^*$-algebra generated by all $H_{m+1}$-translates of $A_{i,m}$. Then
$A_{i,m+1}$ is separable and $H_{m+1}$-invariant. This completes the
recursive construction.

Now, we introduce
$$H:=\bigcup_{m<\omega}H_m$$
and
$$A_i:=\overline{\bigcup_{m<\omega}A_{i,m}}.$$
Then $H$ is countable, and each $A_i$ is a separable unital
$H$-invariant $C^*$-subalgebra of $C(Y_i,\mathbb C)$. Let
$$q_i:Y_i\longrightarrow Y_i':=\operatorname{Spec}(A_i)$$
be the associated factor map of $H$-flows, with metrizable target $Y_i'$, and put
$$q:=(q_1,\ldots,q_n)|_W,\qquad W':=q[W].$$

\begin{clm}
Every $f_j$ factors through $q$, i.e. $f_j = \hat{f}_j \circ q$ for a unique function $\hat{f}_j \in C(W',\mathbb{C})$.
\end{clm}

\begin{clmproof}
Suppose that $q(w)=q(w')$.
Then, for every $i$ and every $a\in A_i$, the function $a$ has the same
value on the $i$-th coordinates of $w$ and $w'$. In particular, this holds
for every member of $\mathcal D_i$. Thus, all the chosen finite sums of
products of coordinate functions have the same value at $w$ and $w'$.
Passing to the uniform limit gives $f_j(w)=f_j(w')$.
Since $q:W\to W'$ is a continuous surjection from a compact space onto a Hausdorff space, it is a quotient map. Hence, there is a unique function $\widehat f_j\in C(W',\mathbb C)$ such that $f_j=\widehat f_j\circ q$.
\end{clmproof}

\begin{clm}
$W'$ is minimal as an $H$-flow.
\end{clm}

\begin{clmproof}
    
By Gelfand duality, the inclusions $A_{i,m}\subseteq A_i$
induce factor maps $r_{i,m}:Y_i'\longrightarrow Y_{i,m}$
and hence continuous maps $r_m:W'\longrightarrow W_m$
such that
$$
q_m=r_m\circ q.
$$

The algebra
$$
\bigcup_{m<\omega}r_m^*C(W_m,\mathbb C)
$$
is uniformly dense in $C(W',\mathbb C)$. Indeed, the coordinate
functions coming from the algebras $A_i$ separate points of $W'$, so
the self-adjoint unital algebra generated by them is dense in
$C(W',\mathbb C)$ by the complex Stone--Weierstrass theorem. Every
member of $A_i$ is a uniform limit of members of the algebras
$A_{i,m}$. By the isometry of the Gelfand transform, such approximation
in $A_i$ yields uniform approximation of the corresponding coordinate
functions on $Y_i'$. 
Finally, it is easy to see that the
coordinate functions coming from the algebras $A_{i,m}$ belong to
$$\bigcup_{m<\omega}r_m^*C(W_m,\mathbb C).$$
It follows that the sets
$r_m^{-1}[B]$, with $m<\omega$ and $B\in\mathcal{B}_m$,
form a basis of the topology of $W'$. 
Explicitly, let
$z\in O\subseteq W'$,
where $O$ is open. Choose
$$\phi\in C(W',[0,1])$$
such that
$\phi(z)=1$ and $\phi=0\text{ on }W'\setminus O$.
By the preceding density, choose $m$ and
$$\chi_m\in C(W_m,\mathbb C)$$
such that
$\left\|\phi-\chi_m\circ r_m\right\|_\infty<\frac13$.
Let $\psi_m:=\operatorname{Re}(\chi_m)\in C(W_m,\mathbb R)$.
Since $\phi$ is real-valued,
$$\left\|\phi-\psi_m\circ r_m\right\|_\infty<\frac13.$$
Therefore,
$$z\in r_m^{-1}\left[\left\{y\in W_m:\psi_m(y)>\frac23\right\}\right]
\subseteq O.$$
Choosing a basis element $B\in\mathcal B_m$ contained 
in the open set $\{y \in W_m: \psi_m(y) >\frac{2}{3}\}$
and containing $r_m(z)$, we get $z\in r_m^{-1}[B]\subseteq O$.

For such $B\in\mathcal{B}_m$, we already fixed a finite set
$F_B\subseteq H$
such that
$$W=\bigcup_{g\in F_B}gq_m^{-1}[B].$$
Since $q$ is $H$-equivariant and $q_m=r_m\circ q$, applying $q$ yields
$$W'=\bigcup_{g\in F_B}g\,r_m^{-1}[B]\subseteq\bigcup_{g\in F_B}gO.$$
Thus, finitely many $H$-translates of every nonempty open subset of $W'$
cover $W'$. Hence, every $H$-orbit in $W'$ is dense, and $W'$ is minimal.
\end{clmproof}

For each $i$, the coordinate map
$$\operatorname{pr}_i|_{W'}:W'\longrightarrow Y_i'$$
is continuous and $H$-equivariant. 
It is also onto as $W'=q[W] = (q_1,\dots,q_n)[W]$ and we have $\pr_i[W']=q_i[\pr_i[W]]= q_i[Y_i]=Y_i'$. 
Therefore, by Claim 2, $Y_i'$ is an
$H$-factor of the minimal flow $W'$, and hence $Y_i'$ is minimal.

Since $Y_i$ is a subflow of the tame $G$-flow $X$, the restricted $H$-flow
$Y_i$ is tame. Tameness passes to factors, so $Y_i'$ is tame. Moreover,
$$\mu_i':=(q_i)_*\mu_i$$
is an $H$-invariant regular Borel probability measure on $Y_i'$.
Lemma \ref{lem:old_claim1} applied to $W'\subseteq Y_1'\times\cdots\times Y_n'$,
gives an $H$-invariant regular Borel probability measure
$\lambda$ on $W'$.

Define a state on the unital $C^*$-subalgebra $q^*C(W',\mathbb C)\subseteq C(W,\mathbb C)$ by
$$L(\varphi\circ q):=\int_{W'}\varphi\,d\lambda.$$
Extend $L$ to a state on $C(W,\mathbb C)$, and let
$\nu$ be its Riesz measure on $W$. Then $q_*\nu=\lambda$.

Using the factorization $f_j=\widehat f_j\circ q$ from Claim 1,
the $H$-equivariance of $q$, the inclusion $g_j\in H$, and the
$H$-invariance of $\lambda$, we obtain
\begin{IEEEeqnarray*}{rCcCl}
\int_W f_j(g_jw)\,d\nu(w) &=& \int_{W'}\widehat f_j(g_jz)\,d\lambda(z)
&=& \int_{W'}\widehat f_j(z)\,d\lambda(z)\\
&=& \int_W f_j(w)\,d\nu(w). & &
\end{IEEEeqnarray*}
Thus, $\nu\in C(g_1,f_1)\cap\cdots\cap C(g_r,f_r)$,
and so we have proved that the family
$$
\left\{
C(g,f):
g\in G,\ f\in C(W,\mathbb C)
\right\}
$$
has nonempty finite intersections. 
By compactness, 
$$
\Prob_G(W)
=
\bigcap_{\substack{g\in G\\ f\in C(W,\mathbb C)}}
C(g,f)
\neq\varnothing.
$$
Thus, $W$ carries a $G$-invariant regular Borel probability measure.
\end{proof}

\begin{cor}
\label{cor:ellis-amenability-characterization}
Let $(X,G)$ be tame. The following conditions are equivalent:
\begin{enumerate}[(i)]
\item $\bigl(E(X,G),G\bigr)$ is hereditarily amenable;

\item $\bigl(E(X,G),G\bigr)$ is amenable;

\item $X$ is hereditarily amenable.
\end{enumerate}
\end{cor}

\begin{proof}
Equivalence between \textup{(i)} and \textup{(ii)} is the content of Remark \ref{rem:ellis-amenability-is-hereditary}.

Assume \textup{(ii)}. By Theorem~\ref{thm:ellis-flow-amenable}, the flow
$X$ is finitely power-hereditarily amenable, and hence is hereditarily
amenable. Thus, \textup{(iii)} holds.

Conversely, \textup{(iii)} implies \textup{(ii)} by
Theorem~\ref{thm:hereditary-amenability-finite-powers}
combined with Theorem~\ref{thm:ellis-flow-amenable}.
\end{proof}

The tameness assumption in
Corollary~\ref{cor:ellis-amenability-characterization} cannot be
omitted: Appendix~\ref{sec:A3} constructs a metrizable
hereditarily amenable flow whose Ellis flow is not amenable.

\subsection{Invariance of Haar decompressions}\label{subsec:Haar.invariance}
The preceding results guarantee the existence of some invariant measure on the Ellis flow. We now prove the stronger statement that the canonical Haar decompression associated with each Ellis group is invariant.
In the first version of this paper, we obtained the result under an additional hypothesis that the flow $(X,G)$ is an inverse limit of metrizable $G$-flows.
Then, during our work in a subsequent project \cite{HoffKru2},
we were able to use a conclusion from \emph{Generalized Newelski Conjecture} (GNC) established there to prove results of this subsection without any metrizability related hypothesis. 
Finally, we obtained arguments not requiring theorems from \cite{HoffKru2} to prove what follows below.  
On the other hand, Lemma~\ref{lemma:K.closure}(1) will be used in \cite{HoffKru2} to give a short proof of GNC.

Now, let $(X,G)$ be a flow, $\CM\unlhd_m E(X,G)$, $u^2=u\in\CM$ and set $K:=u\CM$.

\begin{lemma}\label{lemma:K.closure}
    Assume that $(X,G)$ is tame and $(\CM,G)$ is amenable. 
    Then   
    \begin{enumerate}
    \item $\overline{K}=\CM$,

    \item $\zeta:\CM\to K$, given by $\zeta(\eta)=u\eta$, is continuous with respect to the $\tau$-topology on $K$. 
    \end{enumerate}
\end{lemma}

\begin{proof}
Let $\mu\in\Prob_G(\CM)$. Suppose that $m_0\in\CM\setminus\overline{K}$. There exists $f\in C(\CM, [0,1])$ such that
$f(m_0)=1$ and $f|_{\overline{K}}\equiv 0$. 
In particular, $f(um)=0$ for every $m\in\CM$.
On one hand, we have that $\int_{\CM}f\,d\mu>0$. On the other hand, 
by $G$-invariance of $\mu$ and Lemma \ref{lemma:tameness.dominated.convergence},
we see that
$$0<\int_\CM f(m)\,d\mu(m)= \int_\CM f(m)\,d(u^\#\mu)(m) = \int_{\CM} f(um)\,d\mu(m)=0,$$
a contradiction. Hence (1). Then we obtain (2) by \cite[Lemma 3.1]{KruPiRze}.
\end{proof}

\begin{remark}\label{rem:GPP-failure-of-density}
The amenability assumption in Lemma~\ref{lemma:K.closure} is essential.
For the tame flow studied in
\cite{GismatullinPenazziPillay2015Some}, the corresponding minimal left
ideal $\CM$ and Ellis group $K=u\CM$ satisfy $\overline{K}\subsetneq\CM$.
Consequently, $(\CM,G)$ is not amenable, since otherwise
Lemma~\ref{lemma:K.closure} would give $\overline{K}=\CM$. By
Theorem~\ref{thm:haar-decompression-amenability} below, the associated
Haar decompression $\mu_K$ is therefore not $G$-invariant.
\end{remark}

\begin{theorem}\label{thm:haar-decompression-amenability}
Assume that $(X,G)$ is tame. The following are equivalent:
\begin{enumerate}
    \item The Haar decompression $\mu_K$ is $G$-invariant,

    \item $(\CM,G)$ is amenable,

\end{enumerate}
In particular, if $(\CM,G)$ is amenable, then $\mu_K$ is the unique element of $\Prob_G(\CM)$.
\end{theorem}

\begin{proof}
Assume (1). By Remark \ref{remark: support of the compression}, we have that 
$\supp(\mu_K)\subseteq\CM$. Hence (2).

Now, assume (2), and consider the function $\zeta:\CM\to K$ given by $\zeta(\eta)=u\eta$.
By Lemma \ref{lemma:K.closure} we have that 
$\overline{K}=\CM$ and that $\zeta$ is continuous with respect to the $\tau$-topology on $K$.

Let $\mu\in\Prob_G(\CM)$ and note that $\zeta_*\mu=\mathfrak{h}_K$. 
Indeed, $\CM=E(X,G)u$ and so $Gu$ is dense in $\CM$, thus,
we have that $\zeta[Gu]=\{ugu\,:\,g\in G\}$ is dense in $K$ (with respect to the $\tau$-topology). If $f\in C(K)$ then $f\circ\zeta\in C(\CM)$ and we apply Lemma \ref{lemma:tameness.dominated.convergence} and $G$-invariance of $\mu$
to show that $\zeta_*\mu$ is $\zeta[Gu]$-invariant:
\begin{IEEEeqnarray*}{rCcCl}
    \int\limits_{k\in K}f(ugu\,k)\,d(\zeta_*\mu)(k) &=& \int\limits_{m\in\CM} f(ugu\,um)\,d\mu(m) 
    &=& \int\limits_{m\in\CM} (f\circ\zeta)(gu\,m)\,d\mu(m) \\
    &=& \int\limits_{m\in\CM} (f\circ\zeta)(m)\,d\big((gu)^\#\mu\big)(m) 
    &=& \int\limits_{m\in\CM} (f\circ\zeta)(m)\,d\mu(m) \\
    &=& \int\limits_{k\in K}f(k)\,d(\zeta_*\mu)(k). & &
\end{IEEEeqnarray*}
Therefore, $\zeta_*\mu$ is a $\zeta[Gu]$-invariant
regular Borel probability measure on the compact group $K$, and as such must be equal to its Haar measure $\mathfrak{h}_K$.

Now, we will obtain $G$-invariance of $\mu_K$ by showing that $\mu_K=\mu$.
Let $f\in C(E)$. By Corollary~\ref{cor:Fj.measurable},
$f\circ j$ is measurable for the $\mathfrak h_K$-completion of the
Borel $\sigma$-algebra on $K$. Hence, there are a bounded Borel function
$\widetilde f:K\to\mathbb C$ and a Borel set $N\subseteq K$ such that
$\mathfrak h_K(N)=0$ and $\widetilde f=f\circ j$ on $K\setminus N$.
Since $\zeta_*\mu=\mathfrak h_K$, we have
$\mu(\zeta^{-1}[N])=0$, and therefore
$\widetilde f\circ\zeta=f\circ j\circ\zeta$ $\mu$-almost everywhere.
Moreover, $u^\#\mu=\mu$, since the $G$-invariant measure $\mu$ is fixed
by every element of $E(\Prob(\CM),G)$. Consequently,
using Lemma~\ref{lemma:tameness.dominated.convergence},
\begin{IEEEeqnarray*}{rCcCl}
\int_E f\,d\mu_K
&=& \int_K(f\circ j)(k)\,d\mathfrak h_K(k) &=& \int_K\widetilde f(k)\,d\mathfrak h_K(k)\\
&=& \int_K\widetilde f(k)\,d(\zeta_*\mu)(k) &=& \int_\CM\widetilde f(\zeta(m))\,d\mu(m)\\
&=& \int_\CM f(um)\,d\mu(m) &=& \int_\CM f(m)\,d(u^\#\mu)(m)\\
&=& \int_\CM f(m)\,d\mu(m). & &
\end{IEEEeqnarray*}
Thus, regarding $\mu$ as a probability measure on $E$ concentrated on
$\CM$, we obtain $\mu_K=\mu$.
This ends the proof of (1).

If $\mu'\in\Prob_G(\CM)$ then, by repeating the proof above, we obtain
that $\mu'=\mu_K=\mu$. Thus $\mu_K$ is the unique elment of $\Prob_G(\CM)$.
Alternatively, uniqueness follows from Theorem~\ref{thm:nonmetric-unique-ergodicity}.
\end{proof}

\begin{remark}\label{rem:invariant-decompression-ergodic}
\begin{enumerate}

    \item 
Under the equivalent conditions of
Theorem~\ref{thm:haar-decompression-amenability}, one has
$$\mu_K\in\Erg_G\bigl(E(X,G)\bigr)\qquad\text{and}\qquad
\supp(\mu_K)=\CM=\overline{K}.$$
Indeed, Lemma~\ref{lemma:K.closure} gives
$\overline{K}=\CM$. 
The latter thing follows from Remark~\ref{remark: support of the compression}.
Theorem~\ref{thm:haar-decompression-amenability} shows that
$\mu_K$ is the unique invariant probability measure on $\CM$. Any
convex decomposition of $\mu_K$ into invariant probabilities on
$E(X,G)$ is necessarily concentrated on $\CM$, and uniqueness on
$\CM$ therefore makes the decomposition trivial. Thus, invariance of
an individual Haar decompression already implies its ergodicity.

    \item 
More generally, for tame $(X,G)$,
Haar decompression leads to all ergodic measures on its Ellis flow:
$$\Erg_G(E)=\left\{\mu_{u\CM}:\CM\unlhd_m E,\ 
u^2=u\in\CM,\ (\CM,G)\text{ is amenable}\right\}.$$
The right-to-left inclusion follows by the first point.
Now, if we take $\mu\in\Erg_G(E(X,G))$ then 
Corollary \ref{cor:ergodic-measures-minimal-subflows} implies that
$\supp(\mu)=\CM$ for some minimal left ideal $\CM$.
Because $(\CM,G)$ is amenable, 
we may use Theorem \ref{thm:haar-decompression-amenability}
and obtain that $\mu=\mu_K$ for some Ellis group $K$ of $\CM$.

\item 
Let $(X,G)$ be a tame hereditarily amenable flow and let $K$ and $K'$ be two Ellis groups in $E(X,G)$. Then $\mu_K = \mu_{K'}$ if and only if $K$ and $K'$ are contained in the same minimal left ideal in $E(X,G)$.
This follows by the first point and the uniqueness from Theorem~\ref{thm:haar-decompression-amenability}.
\end{enumerate}
\end{remark}

\begin{cor}\label{cor:global-haar-invariance}
Let $(X,G)$ be tame. The following conditions are equivalent:
\begin{enumerate}[(i)]
\item $(X,G)$ is hereditarily amenable;
\item $(E(X,G),G)$ is amenable;
\item every Haar decompression on $E(X,G)$ is $G$-invariant;
\item some Haar decompression on $E(X,G)$ is $G$-invariant.
\end{enumerate}
\end{cor}

\begin{proof}
By combining
Corollary~\ref{cor:ellis-amenability-characterization}
with Theorem~\ref{thm:haar-decompression-amenability}.
\end{proof}

\section{Ergodic Ellis transfer and realization}
The transfer and realization properties considered in this section
depend on a distinguished point.

It is useful to distinguish two assertions. 
Let $(X,x_0,G)$ be an ambit.
The \emph{ergodic Ellis-transfer property} says that the pushforward under $\ev_{x_0}$ of every Haar decompression from an Ellis group is $G$-invariant and ergodic (e.g.: Lemma \ref{thm:ordinary-ellis-glasner-regime}, Corollary \ref{cor:ergodic-haar-transfer}). The
\emph{ergodic Ellis-realization property} says that every ergodic
$G$-invariant regular Borel probability measure on $X$ is obtained from at
least one Ellis group in this way (e.g.: Corollary \ref{cor:ergodic-ellis-realization}). 
For tame ambits the decompression itself is always defined by
Lemma~\ref{defprop:haar-decompression}. 
The main result of the present section establishes the realization property in the tame setting.

\subsection{Ergodic Ellis-transfer property}
We begin with the transfer problem.
The invariance
criterion from Subsection~\ref{subsec:Haar.invariance}, together with
preservation of ergodicity under equivariant factors, yields the desired
conclusion. Let $(X,G)$ be a flow, put $E:=E(X,G)$, let
$\CM\unlhdm E$, choose an idempotent $u\in\CM$, and put $K:=u\CM$.

\begin{cor}\label{cor:ergodic-haar-transfer}
Let $(X,G)$ be tame and hereditarily amenable.
Then $\mu_K\in\Erg_G(E)$ and, for every $x\in X$, $(\ev_x)_*\mu_K\in\Erg_G(X)$.
In particular, every tame hereditarily amenable ambit has the ergodic
Ellis-transfer property.
\end{cor}

\begin{proof}
Put $E:=E(X,G)$. By
Corollary~\ref{cor:global-haar-invariance}, $\mu_K$ is $G$-invariant.
Remark~\ref{rem:invariant-decompression-ergodic} therefore gives $\mu_K\in\Erg_G(E)$.

Now, fix $x\in X$. Since $\ev_x$ is $G$-equivariant,
$(\ev_x)_*\mu_K$ is $G$-invariant. Moreover, the inverse image under
$\ev_x$ of every $G$-invariant measurable subset of $X$ is a
$G$-invariant measurable subset of $E(X,G)$. Since $\mu_K$ is
ergodic, such an inverse image has $\mu_K$-measure either $0$ or $1$ (by  \cite[Proposition~12.4]{Phelps2001}).
Hence, the $(\ev_x)_*\mu_K$-measure of the original subset of $X$
is either $0$ or $1$ as well, and so $(\ev_x)_*\mu_K$ is ergodic
by Remark~\ref{rem:ergodicity.notnions}.
\end{proof}

\subsection{Ergodic Ellis-realization property}
We now address the complementary realization problem.

\begin{lemma}
\label{lem:ellis-transfer-minimal-support}
Let $(X,G)$ be a tame flow, fix $x_0\in X$, and let $\mu\in \Erg_G(X)$. Assume that
$Y:=\supp(\mu)$ is contained in $\overline{Gx_0}$.
Put $I_Y:=\ev_{x_0}^{-1}[Y]\subseteq E(X,G)$.
Then, for every choice of
$$
\CM\unlhd_m E(X,G),
\qquad
\CM\subseteq I_Y,
\qquad
u^2=u\in\CM,
$$
and $K:=u\CM$, one has $(\ev_{x_0})_*\mu_K=\mu$.
\end{lemma}

\begin{proof}
Let $y_0:=u(x_0)\in Y$. By
Proposition~\ref{prop:tame-ergodic-support-minimal}, the flow
$(Y,G)$ is minimal; it is also tame, and it carries the
invariant probability measure $\mu|_Y$.
Restriction to $Y$ defines a continuous semigroup epimorphism
$$r:E(X,G)\longrightarrow E(Y,G),\qquad r(\eta):=\eta|_Y.$$
Put
$$\CM_Y:=r[\CM],\qquad v:=r(u),\qquad K_Y:=v\CM_Y=r[K].$$
By Fact~\ref{fact:ellis-group-functoriality}, $\CM_Y$ is a minimal
left ideal, $v$ is idempotent, and
$$r|_K:(K,\tau)\longrightarrow(K_Y,\tau)$$
is a continuous group epimorphism satisfying $(r|_K)_*\mathfrak h_K=\mathfrak h_{K_Y}$.

By Corollary~\ref{cor:ergodic-haar-transfer}, the measure
$(\ev_{y_0})_*\mu_{K_Y}$ is $G$-invariant. 
On the other hand,
Theorem~\ref{thm:nonmetric-unique-ergodicity} shows that $(Y,G)$ is
uniquely ergodic. Hence, $(\ev_{y_0})_*\mu_{K_Y}=\mu|_Y$.

Fix
$f\in C(X,\mathbb C)$. Using $ku=k$ for $k\in K$,
$y_0=u(x_0)$, 
$(\ev_{y_0})_* \mu_{K_Y} = \mathfrak{h}_{K_Y}$
and
$(r|_K)_*\mathfrak h_K=\mathfrak h_{K_Y}$, we obtain
\begin{IEEEeqnarray*}{rCcCl}
\int_X f\,d\bigl((\ev_{x_0})_*\mu_K\bigr)
&=& \int_K f\bigl(k(x_0)\bigr)\,d\mathfrak h_K(k)
&=& \int_K f\bigl(k(u(x_0))\bigr)\,d\mathfrak h_K(k)\\
&=& \int_K f\bigl(r(k)(y_0)\bigr)\,d\mathfrak h_K(k)
&=& \int_{K_Y}f\bigl(\ell(y_0)\bigr)\,d\mathfrak h_{K_Y}(\ell)\\
&=& \int_Y f|_Y\,d (\ev_{y_0})_*\mu_{K_Y} &=& \int_X f\,d\mu. 
\end{IEEEeqnarray*}
Since $f\in C(X,\mathbb C)$ was arbitrary, $(\ev_{x_0})_*\mu_K=\mu$.
\end{proof}

\begin{cor}\label{cor:ergodic-ellis-realization}
Every tame ambit has the ergodic Ellis-realization property.
\end{cor}

\begin{proof}
Let $(X,x_0,G)$ be a tame ambit, fix $\mu\in\Erg_G(X)$, and put
$$Y:=\supp(\mu),\qquad I_Y:=\ev_{x_0}^{-1}[Y].$$
Note that $I_Y$ is a left ideal of $E(X,G)$.

Choose a minimal left ideal $\CM\unlhd_m E(X,G)$ contained in $I_Y$,
an idempotent $u\in\CM$, and put $K:=u\CM$. Since
$Y\subseteq X=\overline{Gx_0}$, Lemma~\ref{lem:ellis-transfer-minimal-support}
gives $(\ev_{x_0})_*\mu_K=\mu$.
\end{proof}

For comparison, in the more restrictive one-map setting of tame
$\mathbb N_0$-semicascades, Romanov~\cite{Romanov19} realizes ergodic
measures through generalized ergodic averages in the K\"ohler operator
semigroup, a mechanism different from Haar decompression from Ellis groups.

\subsection{Summary and the Choquet theorem}
Let $(X,x_0,G)$ be a tame hereditarily amenable ambit. Put $E:=E(X,G)$ and
$$\mathscr E(X,G):=\{u\CM:\CM\unlhdm E,\ u^2=u\in\CM\}.$$
By Corollary~\ref{cor:ergodic-haar-transfer}, the map
$\Pi:\mathscr E(X,G)\longrightarrow\Erg_G(X)$, given by $\Pi(K):=(\ev_{x_0})_*\mu_K$,
is well defined, and Corollary~\ref{cor:ergodic-ellis-realization} shows
that it is surjective. Thus,
$$\{(\ev_{x_0})_*\mu_K:K\in\mathscr E(X,G)\}=\Erg_G(X).$$
Here hereditary amenability is used only to ensure that every Ellis
group produces an invariant ergodic measure; 
the ergodic Ellis-realization property itself holds for every tame ambit.

We now assume additionally that $X$ is metrizable and briefly recall the relevant Choquet-theoretic background (cf. \cite{Phelps2001}). Since \(X\)
is compact metrizable, the weak-$*$ compact convex set \(\Prob_G(X)\)
is a metrizable Choquet simplex, and its extreme boundary is precisely
$$\operatorname{ext}\Prob_G(X)=\Erg_G(X).$$
Moreover, \(\Erg_G(X)\) is a \(G_\delta\)-subset of \(\Prob_G(X)\) (cf. \cite{AlfsenBorelSimplex}).
Thus, \(\Prob_G(X)\) is a Bauer simplex precisely when
\(\Erg_G(X)\) is closed, whereas, in the nontrivial case, it is a
Poulsen simplex precisely when \(\Erg_G(X)\) is dense (cf. \cite{LindenstraussOlsenSternfeld}). In particular,
every $\mu\in\Prob_G(X)$ has a unique
representing probability measure $\lambda_\mu$ on $\Erg_G(X)$
\cite[Proposition~3.1 and Chapters~10 and~12]{Phelps2001}.

Equip \(\mathscr E(X,G)\) with the $\sigma$-algebra
$\left\{\Pi^{-1}[B]:B\in\mathcal B(\Erg_G(X))\right\}$.
Since \(\Pi\) is surjective, there is
a unique probability measure \(\widehat\mu\) on
$\mathscr E(X,G)$ determined by $\Pi_*\widehat\mu=\lambda_\mu$.

\begin{cor}
\label{cor:choquet-ellis-decomposition}
Let \((X,x_0,G)\) be a metrizable tame hereditarily amenable ambit.
Then, for every \(\mu\in\Prob_G(X)\),
$$\mu=\int_{\mathscr E(X,G)}(\ev_{x_0})_*\mu_K\,d\widehat\mu(K)$$
in the weak-\(^*\) sense.
\end{cor}

\begin{proof}
    By the Choquet representation theorem and its application to invariant
measures \cite[Proposition~3.1 and Chapters~10 and~12]{Phelps2001},
$\mu$ is the barycenter of its unique representing probability measure
$\lambda_\mu$ on $\Erg_G(X)$.
\end{proof}

\subsection{A link to model theory}
\label{subsec:model.theory}
We conclude by showing that the abstract transfer and realization mechanisms
recover two model-theoretic classifications of ergodic measures. The
first concerns definable groups in NIP theories and the measures of
Chernikov and Simon \cite{ArtemPierre}; the second concerns automorphism flows in amenable NIP theories and the measures of the first author and Rzepecki \cite{HRz}.

\begin{remark} 
Evaluation of Haar decompressions recovers the description of ergodic measures by Chernikov and Simon in the context of definable groups in NIP theories from 
\cite[Theorem 4.5]{ArtemPierre}. 
\end{remark}

\begin{proof}
Consider a definably amenable group $G$ definable in an NIP theory $T$. Let $\mathfrak{C} \models T$ be a monster model, and consider the ambit $(X,x_0,G):=(S_G(\mathfrak{C}),\tp(e/\mathfrak{C}),G)$ (where $G=G(\mathfrak{C})$). Since by \cite{CPS14} amenability transfers to the Shelah expansion, the flow $(S_G(\mathfrak{C}^{\ext}),G)$ is amenable, and hence $(E(S_G(\mathfrak{C}),G),G)$ is amenable as a factor of $(S_G(\mathfrak{C}^{\ext}),G)$. Thus, $(S_G(\mathfrak{C}),\tp(e/\mathfrak{C}),G)$ is hereditarily amenable by Corollary \ref{cor:ellis-amenability-characterization}. This flow is also tame by the NIP assumption (see the introduction in \cite{ArtemPierre}).

As observed above, Corollaries \ref{cor:ergodic-haar-transfer} and \ref{cor:ergodic-ellis-realization} yield the following description of all ergodic measures on $X$:
$$\textrm{Erg}_G(X)=\{(\ev_{x_0})_* \mu_K: K \textrm{ is an Ellis group in } E(X,G)\}.$$

For an $f$-generic type $p \in S_G(\mathfrak{C})$, Definition 3.16 in \cite{ArtemPierre} yields a $G$-invariant measure $\mu_p$ on $X$ which is concentrated on $\overline{Gp}$ by \cite[Remark 3.17]{ArtemPierre}. 
Theorem 4.5 in \cite{ArtemPierre} classifies the ergodic measures on $X$ as all the measures $\mu_p$ for $f$-generic $p$. By the comment after \cite[Proposition 3.31]{ArtemPierre}, instead of $f$-generics we can use here almost periodic types. Our goal is to show that this classification follows from our results. For that it remains to show that 
$$\{\mu_p: p \in X \textrm{ almost periodic}\} = \{(\ev_{x_0})_* \mu_K: K \textrm{ an Ellis group in } E(X,G)\}.$$

For ($\subseteq$) consider any almost periodic $p \in X$. Then there is a minimal left ideal $\mathcal{M}$ in $E(X,G)$ 
such that $p \in \ev_{x_0}[\mathcal{M}] = \mathcal{M}x_0=\mathcal{M} p = \overline{Gp}$. Consider an arbitrary Ellis group $K$ contained in $\mathcal{M}$. By Theorem \ref{thm:haar-decompression-amenability}, 
Remark \ref{rem:invariant-decompression-ergodic}(1),
and Theorem \ref{thm:nonmetric-unique-ergodicity}, 
$(\ev_{x_0})_* \mu_K$ is the unique invariant measure concentrated on the tame minimal flow $\mathcal{M}x_0$.  Since $\mu_p$ is an invariant measure concentrated on $\overline{Gp}=\mathcal{M}x_0$, we get $\mu_p=(\ev_{x_0})_*\mu_K$.

For ($\supseteq$) start from an Ellis group $K$ contained in some minimal left ideal $\mathcal{M}$. Take any $p \in \ev_{x_0}[\mathcal{M}]$. Then $p$ is almost periodic, so the above argument yields $\mu_p=(\ev_{x_0})_*\mu_K$.
\end{proof}

\begin{remark}\label{rem:hoffmann-rzepecki-decompression}
The same mechanism recovers the conclusions of \cite{HRz},
i.e. the description of ergodic measures in the context of countable amenable NIP theory, where the acting group is the group of automorphisms.
\end{remark}

\begin{proof}
Recall the notation of
\cite[Section~5]{HRz}: let $T$ be an amenable NIP theory,
$M\preceq N\preceq\mathfrak C$, let $\bar n$ enumerate $N$, and put
$X:=S_{\bar n}(\mathfrak C)$,  $x_0:=\tp(\bar n/\mathfrak C)$,
$G:=\aut(\mathfrak C)$, and let $I:=\fGen\subseteq S_{\bar n}^{\inv}(\mathfrak C,M)$ be the collection of f-generic types. Because $T$ is NIP, we have that the flow $(X,G)$ is tame (e.g. by \cite{KrRz} or \cite{CodHoff23}). 
Because $T$ is amenable, we have that the flow $(X,G)$ is amenable (by \cite{HruKruPi}).
Every minimal $G$-subflow $Z$ of $X$ is
uniquely ergodic with the unique $G$-invariant Keisler measure $\mu_p$
defined in \cite[Definition~5.2]{HRz} for $p\in Z\subseteq I$
(by \cite[Theorem 5.4]{HRz} and Theorem~\ref{thm:nonmetric-unique-ergodicity}). 
This was established in \cite{HRz} after assuming that $T$ and $M$ are countable (cf. \cite[Corollary~5.14]{HRz}). Moreover, if $T$ and $M$ are countable, then $\{\mu_p:p\in I\}=\Erg_G(X)$ by \cite[Theorem~5.17]{HRz}.
Again, using our more general machinery, we can remove the aforementioned countability assumption of \cite[Theorem~5.17]{HRz} in exchange for working with almost periodic types: 
by Corollary \ref{cor:ergodic-measures-minimal-subflows} and \cite[Theorem 5.4]{HRz} we get
$$\Erg_G(X)=\{\mu_p:p\in X\text{ is almost periodic}\}.$$
Now, we realize measures $\mu_p$ of almost periodic types as Haar decompressions.
Indeed, because flow $(X,G)$ is tame and hereditarily amenable, we may use
Corollaries \ref{cor:ergodic-haar-transfer} and \ref{cor:ergodic-ellis-realization}
and conclude
$$\{\mu_p: p \in X \textrm{ almost periodic}\} = \{(\ev_{x_0})_* \mu_K: K \textrm{ an Ellis group in } E(X,G)\},$$
for every amenable NIP theory $T$ (not necessarily countable).
\end{proof}

\section*{Acknowledgements}
The first author thanks Wei Dai for an insightful discussion of Choquet representations during their meeting at the Tianyuan Mathematics Research Center.

The authors used large language models during manuscript preparation to assist with language editing and typesetting; to search for examples and relevant references; to check minor technical details in proofs; and to test selected arguments and proof strategies by seeking counterexamples. All outputs were independently reviewed and verified by the authors.

In the final stages of manuscript preparation, the authors used Aristotle, developed by Harmonic, to mechanically verify selected statements in Lean~4. The authors take full responsibility for the mathematical content of the paper.

\appendix
\numberwithin{equation}{section}
\section{Examples}\label{sec:A}
We collect three examples. The first shows that the inclusion of an
Ellis group into the enveloping semigroup need not be Borel. The second
gives an explicit split-circle dihedral illustration of
Theorem~\ref{thm:haar-decompression-amenability}: neither Ellis group is
$G$-invariant as a subset of the enveloping semigroup, but their Haar
decompressions coincide and are $G$-invariant. The third shows that
tameness is essential in the characterization of amenability of the
Ellis flow.

\subsection{Non-Borel inclusion of Ellis group}
\label{sec:A1}
Put $\mathbb T:=\mathbb R/\mathbb Z$, equipped with its usual circular
order. For $A\subseteq\mathbb T$, let
$\operatorname{Split}(\mathbb T;A)$ be the circularly ordered space
obtained by replacing each $t\in A$ by two adjacent points
$t^-<t^+$ and leaving every $t\notin A$ unsplit. We use the convention
$t^-=t^+=t$ at unsplit points, and equip
$\operatorname{Split}(\mathbb T;A)$ with the cyclic-order topology,
that is, the topology generated by the open cyclic intervals.

Choose $\alpha,\beta\in\mathbb R$ such that
$1,\alpha,\beta$ are linearly independent over $\mathbb Q$. For
$\xi\in\{\alpha,\beta\}$, put $A_\xi:=\{n\xi:n\in\mathbb Z\}\subseteq\mathbb T$ 
and let $X_\xi:=\operatorname{Split}(\mathbb T;A_\xi)$.
Let $T_\xi(t^\rho):=(t+\xi)^\rho$,
and equip $X:=X_\alpha\times X_\beta$ with the diagonal
$\mathbb Z$-action generated by $T_\alpha\times T_\beta$.

By \cite[Proposition~4.1]{GM}, this is a minimal metrizable tame
system, and
$$\CM:=\left\{\bigl(p^\sigma_{\alpha,\gamma},
      p^\rho_{\beta,\eta}\bigr):\sigma,\rho\in\{+,-\},\
\gamma,\eta\in\mathbb T\right\}$$
is a minimal $\mathbb Z$-subflow of $E(X,\mathbb Z)$, hence a minimal
left ideal. Here $p^\sigma_{\xi,\gamma}(t^\rho):=(t+\gamma)^\sigma$.
The composition rule
$$p^\sigma_{\xi,\gamma}\circ p^\rho_{\xi,\eta}=p^\sigma_{\xi,\gamma+\eta}$$
shows that $u:=\bigl(p^+_{\alpha,0},p^+_{\beta,0}\bigr)$
is an idempotent and that
$$K:=u\CM=\left\{\bigl(p^+_{\alpha,\gamma},p^+_{\beta,\eta}\bigr):
(\gamma,\eta)\in\mathbb T^2\right\}.$$

The natural factor map $X\to\mathbb T^2$ induces a semigroup
epimorphism
$$\pi_E:E(X,\mathbb Z)\longrightarrow\mathbb T^2.$$
By Fact~\ref{fact:ellis-group-functoriality}, its restriction
$\pi_E|_K:(K,\tau)\longrightarrow\mathbb T^2$,
given by $\bigl(p^+_{\alpha,\gamma},p^+_{\beta,\eta}\bigr)
\longmapsto(\gamma,\eta)$,
is a continuous quotient group homomorphism. It is bijective, and
therefore it is a homeomorphism.

On the other hand, \cite[Proposition~4.1]{GM} shows that
$$\Delta:=\left\{\bigl(p^+_{\alpha,\gamma},p^+_{\beta,-\gamma}\bigr):\gamma\in\mathbb T\right\}$$
is discrete in the topology inherited from $E(X,\mathbb Z)$. In the
$\tau$-topology it is a closed ordinary circle. Choose a non-Borel set
$B\subseteq\mathbb T$ and put
$$\Delta_B:=\left\{\bigl(p^+_{\alpha,\gamma},p^+_{\beta,-\gamma}\bigr):
\gamma\in B\right\}.$$
Since $\Delta$ is discrete in the inherited topology, there is an open
set $O\subseteq E(X,\mathbb Z)$ such that
$O\cap\Delta=\Delta_B$. If the inclusion
$j:(K,\tau)\to E(X,\mathbb Z)$ were Borel, then
$j^{-1}[O]\cap\Delta=\Delta_B$ would be Borel in the ordinary circle
$\Delta$, a contradiction.

\subsection{Haar decompression restores invariance}
\label{subsec:split-circle-ordinary}
This example exhibits two Ellis groups which are interchanged by the
reflection, although their Haar decompressions coincide and are
$G$-invariant. 

The \emph{fully split circle} is $X=\mathbb T^{\bowtie}:=
\operatorname{Split}(\mathbb T;\mathbb T)$ (cf. Subsection \ref{sec:A1}).
Thus every $t\in\mathbb T$ is replaced by two adjacent points
$t^-<t^+$. The arcs $[a^+,b^-]$, $a\neq b$, form a clopen basis, and
the map
$\pi:X\longrightarrow\mathbb T$, $\pi(t^\pm):=t$,
is continuous and onto. For $\rho\in\{+,-\}$, set $\rho^1:=\rho$ and
let $\rho^{-1}$ be the opposite sign.
For $\varepsilon\in\{\pm1\}$ and $a\in\mathbb T$, define
$$P_{\varepsilon,a}(t^\rho):=(\varepsilon t+a)^{\rho^\varepsilon}.$$
Then $P_{\sigma,d}\circ P_{\varepsilon,a}=P_{\sigma\varepsilon,d+\sigma a}$.
Fix an irrational $\alpha\in\mathbb T$, put
$D:=\langle\alpha\rangle$, and let
$$
R:=P_{1,\alpha},
\qquad
J:=P_{-1,0},
\qquad
G:=\langle R,J\rangle
=
\{P_{\varepsilon,d}:
  \varepsilon\in\{\pm1\},\ d\in D\}.
$$
Thus, $G\cong D_\infty$ is amenable, and $(X,x_0,G)$ is an ambit for
$x_0:=0^+$. The flow is tame: the index-two subgroup
$\langle R\rangle$ acts tamely by
\cite[Theorem~4.5]{Glasner3}, and the $G$-orbit of every $f\in C(X)$ is
the union of two translates of its $\langle R\rangle$-orbit.

For $\delta\in\{+,-\}$, $\varepsilon\in\{\pm1\}$, and $a\in\mathbb T$,
define the collapsed map
$$
P^\delta_{\varepsilon,a}(t^\rho)
:=
(\varepsilon t+a)^\delta,
\qquad
M_\delta
:=
\{P^\delta_{\varepsilon,a}:
  \varepsilon\in\{\pm1\},\ a\in\mathbb T\}.
$$
The superscript records the side onto which every two-point fibre is
collapsed. The product formulas are
$$P^\delta_{\varepsilon,a}\circ
P^{\delta'}_{\varepsilon',b}=
P^\delta_{\varepsilon\varepsilon',a+\varepsilon b},\quad
P_{\sigma,d}\circ P^\delta_{\varepsilon,a}=
P^{\delta^\sigma}_{\sigma\varepsilon,d+\sigma a},\quad
P^\delta_{\varepsilon,a}\circ P_{\sigma,d}=
P^\delta_{\varepsilon\sigma,a+\varepsilon d}.$$

\begin{lemma}
\label{prop:sc-E-classification}
One has
$$
E(X,G)=G\,\dot\cup\,M_+\,\dot\cup\,M_-.
$$
Moreover, $M:=M_+\cup M_-$ is the unique minimal left ideal. For
$\delta\in\{+,-\}$,
$$
u_\delta:=P^\delta_{1,0}
\qquad\text{and}\qquad
K_\delta:=u_\delta M=M_\delta.
$$
Then $u_\delta$ is idempotent, and
$$
(K_\delta,\tau)
\longrightarrow
O(2)\cong\{\pm1\}\ltimes\mathbb T,
\qquad
P^\delta_{\varepsilon,a}\longmapsto(\varepsilon,a),
$$
is an isomorphism of compact Hausdorff topological groups.
\end{lemma}

\begin{proof}
If $d_i\in D$ approaches $a$ strictly from side $\delta$, then
$$
P_{\varepsilon,d_i}
\longrightarrow
P^\delta_{\varepsilon,a}
$$
pointwise on $X$. Conversely, after passing to a subnet, every net of
principal maps has constant orientation and its translation parameters are
either eventually constant or approach their limit strictly from one side.
This gives the classification and shows that, for each fixed $\varepsilon\in\{\pm1\}$, the map
$$
\mathbb T^{\bowtie}\longrightarrow E(X,G),
\qquad
a^\delta\longmapsto P^\delta_{\varepsilon,a},
$$
is a homeomorphism onto its image.

The product formulas imply that $M$ is a left ideal
and that every nonempty left ideal contains $M$, and they give the asserted
idempotents and Ellis groups. The factor map $\pi$ induces a continuous
semigroup epimorphism
$$
\pi_E:E(X,G)\longrightarrow O(2)
$$
which sends both $P_{\varepsilon,a}$ and
$P^\delta_{\varepsilon,a}$ to $(\varepsilon,a)$. Its restriction to
$K_\delta$ is a continuous bijective homomorphism of compact Hausdorff
groups, and hence is an isomorphism.
\end{proof}

\begin{lemma}
\label{lem:sc-countable-jumps}
For every $f\in C(X,\mathbb C)$, the set
$$
J_f
:=
\{a\in\mathbb T:f(a^+)\neq f(a^-)\}
$$
is countable. The two sections
$$
a\longmapsto a^+
\qquad\text{and}\qquad
a\longmapsto a^-
$$
are Borel.
\end{lemma}

\begin{proof}
For $n\geqslant1$, put
$$
J_{f,n}
:=
\left\{
a\in\mathbb T:
|f(a^+)-f(a^-)|\geqslant\frac1n
\right\}.
$$
If $J_{f,n}$ were infinite, it would contain a sequence $(a_i)$ of
distinct points converging to some $a\in\mathbb T$. Passing to a
subsequence, all $a_i$ approach $a$ from the same side. Then the pairs
$a_i^-,a_i^+$ converge to the same one of $a^-,a^+$, contradicting
continuity of $f$. Thus, each $J_{f,n}$ is finite, and
$J_f=\bigcup_{n\geqslant1}J_{f,n}$
is countable. The Borel assertion follows directly by taking preimages of
the clopen basis arcs $[a^+,b^-]$.
\end{proof}

Let $\mathfrak h_\delta$ be the normalized Haar measure on $K_\delta$,
and let $\mu_\delta$ be its decompression. Let $da$ denote normalized
Haar measure on $\mathbb T$. Under the identification
$$
K_\delta\cong\{\pm1\}\ltimes\mathbb T,
\qquad
P^\delta_{\varepsilon,a}\longmapsto(\varepsilon,a),
$$
the measure $\mathfrak h_\delta$ corresponds to
$\frac12 \sum_{\varepsilon\in\{\pm1\}} \delta_\varepsilon\otimes da$.

\begin{proposition}
\label{thm:sc-ordinary-decompression}
For every $F\in C(E(X,G),\mathbb C)$,
\begin{equation}
\label{eq:sc-decompression}
\int_{E(X,G)}F\,d\mu_\delta
=
\frac12
\sum_{\varepsilon\in\{\pm1\}}
\int_{\mathbb T}
F(P^\delta_{\varepsilon,a})\,da.
\end{equation}
The two decompressions coincide,
$\mu_+=\mu_-=:\mu_E,$
and $\mu_E$ is $G$-invariant.

The ambit $(X,x_0,G)$ is uniquely ergodic. Its unique invariant probability
measure $\nu_X$ satisfies
\begin{equation}
\label{eq:sc-nuX}
\int_X\varphi\,d\nu_X
=
\frac12
\int_{\mathbb T}
\bigl(\varphi(a^+)+\varphi(a^-)\bigr)\,da
\end{equation}
where $\varphi\in C(X,\mathbb C)$, and we have
$(\ev_{x_0})_*\mu_+=(\ev_{x_0})_*\mu_-=\nu_X$.
\end{proposition}

\begin{proof}
The identification $K_\delta\cong O(2)$ gives
\eqref{eq:sc-decompression}. For fixed $\varepsilon$, the restriction of
$F$ to the double circle
$$
\{P^\delta_{\varepsilon,a}:
  \delta\in\{+,-\},\ a\in\mathbb T\}
$$
is continuous. Lemma~\ref{lem:sc-countable-jumps} therefore shows that
$F(P^+_{\varepsilon,a})=F(P^-_{\varepsilon,a})$
for Lebesgue-almost every $a$, and hence
$\mu_+=\mu_-$.
The product formulas imply
$$
(P_{\sigma,d}\circ -)_*\mu_\delta
=
\mu_{\delta^\sigma},
$$
so the common decompression is $G$-invariant.

Since
$P^\delta_{\varepsilon,a}(x_0)=a^\delta$,
formula~\eqref{eq:sc-decompression} and Lemma \ref{lem:sc-countable-jumps} give the evaluation identity and \eqref{eq:sc-nuX}.

Finally, let $\nu$ be any $G$-invariant probability measure on $X$. It is
atomless because all $R$-orbits are infinite, and $\pi_*\nu$ is Lebesgue
measure by unique ergodicity of the irrational rotation. For
$\varphi\in C(X)$, put
$\bar\varphi(a):=\varphi(a^+)$.
The function $\bar\varphi$ is Borel, and
$$
\varphi(x)
=
\bar\varphi(\pi(x))
$$
outside the countable union of fibres over $J_\varphi$. Since $\nu$ is
atomless, this exceptional set is $\nu$-null. Thus,
$$
\int_X\varphi\,d\nu
=
\int_{\mathbb T}\bar\varphi(a)\,da,
$$
which is independent of $\nu$. Hence, $\nu=\nu_X$.
\end{proof}

We emphasize that the individual Ellis groups are not $G$-invariant:
the rotation preserves each of them, whereas the reflection interchanges
them,
$$
R\circ K_\delta=K_\delta,
\qquad
J\circ K_\delta=K_{\delta^{-1}}.
$$
Nevertheless, their common decompression
$\mu_E=(j_+)_*\mathfrak h_+=(j_-)_*\mathfrak h_-$
is $G$-invariant.

\subsection{Necessity of tameness}
\label{sec:A3}
The following example shows that hereditary amenability alone does not
imply amenability of the Ellis flow, so the tameness hypothesis in
Theorem~\ref{thm:hereditary-amenability-finite-powers} and
Corollary~\ref{cor:ellis-amenability-characterization} is indispensable.

\begin{proposition}
\label{prop:hereditarily-amenable-nonamenable-ellis}
There is a metrizable hereditarily amenable flow $(X,G)$ such
that $(E(X,G),G)$ is not amenable.
\end{proposition}

\begin{proof}
Let $H:=F_2=\langle a,b\rangle$ and $S:=\{a,a^{-1},b,b^{-1}\}$.
We identify the boundary $\partial H$ with the compact space of
infinite reduced words in the alphabet $S$, equipped with the cylinder
topology and the natural left $H$-action. For $s\in S$, let
$[s]\subseteq\partial H$ denote the cylinder of words beginning with
$s$.

Put $C_2:=\mathbb Z/2\mathbb Z$, $X:=C_2^{H\times S}$,
and let $H$ act by
$$(h\cdot x)(q,s):=x(h^{-1}q,s).$$
Let
$$D:=\bigoplus_{(q,s)\in H\times S} C_2=\{d\in X:\supp(d)\text{ is finite}\}$$
be the dense subgroup of finitely supported functions, and put $G:=D\rtimes H$.
The affine action is
$$(d,h)\cdot x:=d+h\cdot x.$$
Every $D$-orbit is dense in $X$, so the $G$-flow $X$ is minimal.
The normalized Haar measure on the compact group $X$ is invariant under
translations by $D$ and under the coordinate permutations induced by
$H$. Thus the $G$-flow $X$ is amenable. 
Since it is minimal, its only nonempty
subflow is $X$ itself, and hence $(X,G)$ is hereditarily amenable.

Define
$$\iota:\partial H\longrightarrow X,
\qquad
\iota(\xi)(q,s):=\mathbf 1_{[s]}(q^{-1}\xi).$$
This map is continuous and $H$-equivariant. 
It is easy too see that $\iota$ is injective; since
$\partial H$ is compact and $X$ is Hausdorff, it is an embedding.
Put $Y:=\iota[\partial H]$.
The $H$-flow $Y\cong\partial H$ has no invariant probability measure.
Indeed, for every $s\in S$,
$$s^{-1}[s]=\partial H\setminus[s^{-1}].$$
If $\lambda$ were an invariant Borel probability measure, 
then $\lambda([s])+\lambda([s^{-1}])=1$.
Taking $s=a$ and $s=b$ would make the sum of the measures of the four
first-letter cylinders equal to $2$, although these cylinders
partition $\partial H$.

Now, we give $Y$ the $G$-action induced by the quotient
$G\twoheadrightarrow H$, so that $D$ acts trivially. Define
$$\delta:X^2\longrightarrow X,\qquad\delta(x,y):=y-x,$$
and put $W:=\delta^{-1}[Y]$.
For $(d,h)\in G$,
$$\delta\bigl((d,h)\cdot x,(d,h)\cdot y\bigr)=h\cdot\delta(x,y).$$
Therefore, $W$ is a $G$-subflow and $\delta|_W:W\longrightarrow Y$
is a continuous $G$-map. It is onto, since
$\delta(0,y)=y$ for every $y\in Y$. Thus it is a factor map.

If $W$ carried a $G$-invariant probability measure, its pushforward
under $\delta|_W$ would be a $G$-invariant probability measure on
$Y$, a contradiction. Hence, the second power $X^2$ contains a
nonamenable subflow, so $(X,G)$ is not finitely
power-hereditarily amenable. By
Theorem~\ref{thm:ellis-flow-amenable}, the Ellis flow
$(E(X,G),G)$ is not amenable. In particular, $(X,G)$ is not tame by Corollary \ref{cor:ellis-amenability-characterization}.
\end{proof}

\printbibliography

@article{newelski09,
  title={Topological dynamics of definable group actions},
  author={Newelski, Ludomir},
  journal={The Journal of Symbolic Logic},
  doi={10.2178/jsl/1231082302},
  volume={74},
  number={1},
  pages={50--72},
  year={2009},
  publisher={Cambridge University Press}
}

@article{newelski12,
  title={Model theoretic aspects of the {E}llis semigroup},
  author={Newelski, Ludomir},
  journal={Israel Journal of Mathematics},
  volume={190},
  number={1},
  pages={477--507},
  year={2012},
  publisher={Springer},
  doi={10.1007/s11856-011-0202-6},
}

@article{KruPil17,
	title        = {Generalised {B}ohr compactification and model-theoretic connected components},
	author       = {Krupiński, Krzysztof and Pillay, Anand},
	year         = 2017,
	journal      = {Mathematical Proceedings of the Cambridge Philosophical Society},
	publisher    = {Cambridge University Press},
	volume       = 163,
	number       = 2,
	pages        = {219–249},
	doi          = {10.1017/S0305004116000967}
}

@article{GM,
  author       = {Glasner, Eli and Megrelishvili, Michael},
  title        = {Todor{\v{c}}evi{\'c}'s trichotomy and a hierarchy in the class of tame dynamical systems},
  journaltitle = {Trans. Amer. Math. Soc.},
  volume       = {375},
  number       = {7},
  year         = {2022},
  pages        = {4513--4548},
  doi          = {10.1090/tran/8522}
}

@phdthesis{rzepecki2018,
  author      = {Rzepecki, Tomasz},
  title       = {Bounded Invariant Equivalence Relations},
  institution = {University of Wroc{\l}aw},
  location    = {Wroc{\l}aw},
  year        = {2018},
  eprinttype  = {arXiv},
  eprint     = {1810.05113}
}

@article{GlasnerMegrelishviliUspenskij,
  author       = {Glasner, Eli and Megrelishvili, Michael and Uspenskij, Vladimir V.},
  title        = {On metrizable enveloping semigroups},
  journaltitle = {Israel J. Math.},
  volume       = {164},
  year         = {2008},
  pages        = {317--332},
  doi          = {10.1007/s11856-008-0032-3}
}

@book{Auslander,
  author    = {Auslander, Joseph},
  title     = {Minimal Flows and Their Extensions},
  series    = {North-Holland Mathematics Studies},
  volume    = {153},
  publisher = {North-Holland},
  location  = {Amsterdam},
  year      = {1988}
}

@incollection{GM23,
  author       = {Glasner, Eli and Megrelishvili, Michael},
  title        = {More on tame dynamical systems},
  booktitle    = {Ergodic Theory and Dynamical Systems in Their Interactions with Arithmetics and Combinatorics},
  editor       = {Ferenczi, S{\'e}bastien and Ku{\l}aga-Przymus, Joanna and Lema{\'n}czyk, Mariusz},
  series       = {Lecture Notes in Mathematics},
  volume       = {2213},
  publisher    = {Springer},
  location     = {Cham},
  year         = {2018},
  pages        = {351--392},
  doi          = {10.1007/978-3-319-74908-2_18}
}

@article{Iba16,
	title        = {The dynamical hierarchy for {R}oelcke precompact {P}olish groups},
	author       = {Ibarlucía, Tom\'{a}s},
	year         = 2016,
	journal      = {Israel J. Math.},
	volume       = 215,
	number       = 2,
	pages        = {965--1009},
	doi          = {10.1007/s11856-016-1399-1},
	issn         = {0021-2172,1565-8511},
	fjournal     = {Israel Journal of Mathematics},
	mrclass      = {03C45 (03C15 03C50 22A10 37B05 37C85)},
	mrnumber     = 3552300,
	mrreviewer   = {G.\ Cherlin}
}

@article{KrRz,
  author       = {Krupi{\'n}ski, Krzysztof and Rzepecki, Tomasz},
  title        = {Galois groups as quotients of {P}olish groups},
  journaltitle = {J. Math. Log.},
  volume       = {20},
  number       = {3},
  year         = {2020},
  eid          = {2050018},
  pagetotal    = {48},
  doi          = {10.1142/S021906132050018X}
}

@article{KerrLi,
  author       = {Kerr, David and Li, Hanfeng},
  title        = {Independence in topological and {$C^*$}-dynamics},
  journaltitle = {Math. Ann.},
  volume       = {338},
  number       = {4},
  year         = {2007},
  pages        = {869--926},
  doi          = {10.1007/s00208-007-0097-z}
}

@article{Glasner1,
  author       = {Glasner, Eli},
  title        = {On tame dynamical systems},
  journaltitle = {Colloq. Math.},
  volume       = {105},
  number       = {2},
  year         = {2006},
  pages        = {283--295},
  doi          = {10.4064/cm105-2-9}
}

@article{Glasner2,
  author       = {Glasner, Eli},
  title        = {The structure of tame minimal dynamical systems for general groups},
  journaltitle = {Invent. Math.},
  volume       = {211},
  number       = {1},
  year         = {2018},
  pages        = {213--244},
  doi          = {10.1007/s00222-017-0747-z}
}

@article{GLASNER_2007,
  author       = {Glasner, Eli},
  title        = {The structure of tame minimal dynamical systems},
  journaltitle = {Ergodic Theory Dynam. Systems},
  volume       = {27},
  number       = {6},
  year         = {2007},
  pages        = {1819--1837},
  doi          = {10.1017/S0143385707000296}
}

@article{Glasner3,
  author       = {Glasner, Eli and Megrelishvili, Michael},
  title        = {Circularly ordered dynamical systems},
  journaltitle = {Monatsh. Math.},
  volume       = {185},
  number       = {3},
  year         = {2018},
  pages        = {415--441},
  doi          = {10.1007/s00605-017-1134-y}
}

@article{Huang,
  author       = {Huang, Wen},
  title        = {Tame systems and scrambled pairs under an abelian group action},
  journaltitle = {Ergodic Theory Dynam. Systems},
  volume       = {26},
  number       = {5},
  year         = {2006},
  pages        = {1549--1567},
  doi          = {10.1017/S0143385706000198}
}

@article{Simon_2015,
  author       = {Simon, Pierre},
  title        = {Rosenthal compacta and {NIP} formulas},
  journaltitle = {Fund. Math.},
  volume       = {231},
  number       = {1},
  year         = {2015},
  pages        = {81--92},
  doi          = {10.4064/fm231-1-5}
}

@article{ArtemPierre,
	title        = {Definably amenable {NIP} groups},
	author       = {Chernikov, Artem and Simon, Pierre},
	year         = 2018,
	journal      = {J. Amer. Math. Soc.},
	volume       = 31,
	number       = 3,
	pages        = {609--641},
	doi          = {10.1090/jams/896},
	issn         = {0894-0347,1088-6834},
	fjournal     = {Journal of the American Mathematical Society},
	mrclass      = {03C45 (03C60 03C64 22F10 28D15 37B05)},
	mrnumber     = 3787403,
	mrreviewer   = {Rahim\ Nazim\ Moosa}
}

@article{HruKruPi,
  author       = {Hrushovski, Ehud and Krupi{\'n}ski, Krzysztof and Pillay, Anand},
  title        = {On first order amenability},
  journaltitle = {Selecta Math. (N.S.)},
  volume       = {32},
  year         = {2026},
  eid          = {23},
  doi          = {10.1007/s00029-026-01125-1}
}

@book{Kechris95,
  author    = {Kechris, Alexander S.},
  title     = {Classical Descriptive Set Theory},
  series    = {Graduate Texts in Mathematics},
  volume    = {156},
  publisher = {Springer-Verlag},
  location  = {New York},
  year      = {1995},
  doi       = {10.1007/978-1-4612-4190-4}
}

@article{KruPi23,
  author  = {Krupi{\'n}ski, Krzysztof and Pillay, Anand},
  title   = {Generalized Locally Compact Models for Approximate Groups},
  journal = {Adv. Math.},
  volume  = {502},
  year    = {2026},
  pages   = {111154},
  doi     = {https://doi.org/10.1016/j.aim.2026.111154}
}

@article{KruPiRze,
  author       = {Krupi{\'n}ski, Krzysztof and Pillay, Anand and Rzepecki, Tomasz},
  title        = {Topological dynamics and the complexity of strong types},
  journaltitle = {Israel J. Math.},
  volume       = {228},
  number       = {2},
  year         = {2018},
  pages        = {863--932},
  doi          = {10.1007/s11856-018-1780-3}
}

@article{CPS14,
author = {Chernikov, Artem and Pillay, Anand and Simon, Pierre},
title = {External definability and groups in NIP theories},
journal = {Journal of the London Mathematical Society},
volume = {90},
number = {1},
pages = {213-240},
doi = {10.1112/jlms/jdu019},
year = {2014}
}

@article{Krupinski2019,
  author       = {Krupi{\'n}ski, Krzysztof and Newelski, Ludomir and Simon, Pierre},
  title        = {Boundedness and absoluteness of some dynamical invariants in model theory},
  journaltitle = {J. Math. Log.},
  volume       = {19},
  number       = {2},
  year         = {2019},
  eid          = {1950012},
  pagetotal    = {55},
  doi          = {10.1142/S0219061319500120}
}

@article{Krupinski2019a,
  author       = {Krupi{\'n}ski, Krzysztof and Pillay, Anand},
  title        = {Amenability, definable groups, and automorphism groups},
  journaltitle = {Adv. Math.},
  volume       = {345},
  year         = {2019},
  pages        = {1253--1299},
  doi          = {10.1016/j.aim.2019.01.033}
}

@book{Glasner1976,
  author    = {Glasner, Shmuel},
  title     = {Proximal Flows},
  series    = {Lecture Notes in Mathematics},
  volume    = {517},
  publisher = {Springer-Verlag},
  location  = {Berlin},
  year      = {1976},
  doi       = {10.1007/BFb0080139}
}

@article{GismatullinPenazziPillay2015Some,
  author  = {Gismatullin, Jakub and Penazzi, Davide and Pillay, Anand},
  title   = {Some model theory of {$\mathrm{SL}(2,\mathbb{R})$}},
  journal = {Fundamenta Mathematicae},
  year    = {2015},
  volume  = {229},
  number  = {2},
  pages   = {117--128},
  doi     = {10.4064/fm229-2-2}
}

@book{LMNS10,
  author    = {Luke{\v{s}}, Jaroslav and Mal{\'y}, Jan and Netuka, Ivan and Spurn{\'y}, Ji{\v{r}}{\'i}},
  title     = {Integral Representation Theory: Applications to Convexity, Banach Spaces and Potential Theory},
  series    = {De Gruyter Studies in Mathematics},
  volume    = {35},
  publisher = {Walter de Gruyter},
  location  = {Berlin},
  year      = {2010},
  doi       = {10.1515/9783110203219}
}

@article{FGJO21,
  author  = {Fuhrmann, Gabriel and Glasner, Eli and J{\"a}ger, Tobias
             and Oertel, Christian},
  title   = {Irregular model sets and tame dynamics},
  journal = {Transactions of the American Mathematical Society},
  year    = {2021},
  volume  = {374},
  number  = {5},
  pages   = {3703--3734},
  doi     = {10.1090/tran/8349}
}

@article{Romanov16,
  author       = {Romanov, Alexey V.},
  title        = {Ergodic properties of discrete dynamical systems and enveloping semigroups},
  journaltitle = {Ergodic Theory Dynam. Systems},
  volume       = {36},
  number       = {1},
  year         = {2016},
  pages        = {198--214},
  doi          = {10.1017/etds.2014.62}
}

@article{Romanov19,
  author       = {Romanov, Alexey V.},
  title        = {Ergodic properties of tame dynamical systems},
  journaltitle = {Math. Notes},
  volume       = {106},
  year         = {2019},
  pages        = {286--295},
  doi          = {10.1134/S0001434619070319}
}

@book{Phelps2001,
  author    = {Phelps, Robert R.},
  title     = {Lectures on Choquet's Theorem},
  edition   = {2},
  series    = {Lecture Notes in Mathematics},
  volume    = {1757},
  publisher = {Springer},
  location  = {Berlin},
  year      = {2001},
  doi       = {10.1007/b76887}
}

@article{CGK,
    author  = {Chernikov, Artem and Gannon, Kyle and Krupi{\'n}ski, Krzysztof},
    title   = {Definable Convolution and Idempotent Keisler Measures III: Generic Stability, Generic Transitivity, and Revised Newelski's Conjecture},
    journal = {J. Lond. Math. Soc.},
    volume  = {114},
    number  = {1},
    year    = {2026},
    pages   = {e70639},
    doi     = {10.1112/jlms.70639}
}

@unpublished{BaZu,
  author     = {Basso, Gianluca and Zucker, Andy},
  title      = {Topological groups with tractable minimal dynamics},
  year       = {2024},
  eprinttype = {arXiv},
  eprint     = {2412.05659},
  note       = {Submitted}
}

@online{CodHoff23,
  author     = {Codenotti, Alessandro and Hoffmann, Daniel Max},
  title      = {Ranks in {E}llis semigroups and model theory},
  year       = {2023},
  eprinttype = {arXiv},
  eprint     = {2308.05477}
}

@online{HRz,
  author       = {Hoffmann, Daniel Max and Rzepecki, Tomasz},
  title        = {On idempotent measure conjecture and decomposition of
                  invariant measures},
  year         = {2025},
  eprinttype   = {arxiv},
  eprint       = {2511.22945}
}

@article{AlfsenBorelSimplex,
  author  = {Alfsen, Erik M.},
  title   = {A note on the {B}orel structure of a metrizable
             {C}hoquet simplex and of its extreme boundary},
  journal = {Mathematica Scandinavica},
  volume  = {19},
  year    = {1966},
  pages   = {161--171},
  doi     = {10.7146/math.scand.a-10805}
}

@article{LindenstraussOlsenSternfeld,
  author  = {Lindenstrauss, Joram and Olsen, Gunnar and Sternfeld, Y.},
  title   = {The {Poulsen} simplex},
  journal = {Annales de l'Institut Fourier},
  volume  = {28},
  number  = {1},
  year    = {1978},
  pages   = {91--114},
  doi     = {10.5802/aif.682}
}

@article{KLM21,
  author       = {Krupi{\'n}ski, Krzysztof and Lee, Junguk and Moconja, Slavko},
  title        = {Ramsey theory and topological dynamics for first order theories},
  journaltitle = {Trans. Amer. Math. Soc.},
  volume       = {375},
  number       = {4},
  year         = {2022},
  pages        = {2553--2596},
  doi          = {10.1090/tran/8594}
}

@article{GlasnerMegrelishvili2012,
  author  = {Glasner, Eli and Megrelishvili, Michael},
  title   = {Representations of dynamical systems on Banach spaces
             not containing {$\ell_1$}},
  journal = {Trans. Amer. Math. Soc.},
  volume  = {364},
  number  = {12},
  pages   = {6395--6424},
  year    = {2012},
  doi     = {10.1090/S0002-9947-2012-05549-8}
}

@book{MurphyCstar,
  author    = {Murphy, Gerard J.},
  title     = {{$C^*$-Algebras and Operator Theory}},
  publisher = {Academic Press},
  location  = {Boston},
  year      = {1990}
}

@incollection{GlasnerMegrelishvili2013,
  author    = {Glasner, Eli and Megrelishvili, Michael},
  title     = {Banach Representations and Affine Compactifications of Dynamical Systems},
  booktitle = {Asymptotic Geometric Analysis: Proceedings of the Fall 2010 Fields Institute Thematic Program},
  editor    = {Ludwig, Monika and Milman, Vitali D. and Pestov, Vladimir and Tomczak-Jaegermann, Nicole},
  series    = {Fields Institute Communications},
  volume    = {68},
  pages     = {75--144},
  publisher = {Springer},
  location  = {New York},
  year      = {2013},
  doi       = {10.1007/978-1-4614-6406-8_6},
  eprinttype = {arxiv},
  eprint    = {1204.0432}
}

@unpublished{HoffKru2,
  author     = {Hoffmann, Daniel Max and Krupiński, Krzysztof},
  title      = {On tame amenable flows and generalized Newelski's conjecture},
  year       = {2026},
  note       = {in preparation}
}
\end{document}